\documentclass[12pt]{article}
\usepackage[a4paper,
            bindingoffset=0.2in,
            left=0.75in,
            right=0.75in,
            top=1.5in,
            bottom=1.5in,
            footskip=.25in]{geometry}
\usepackage[utf8]{inputenc}
\usepackage{amsmath}
\usepackage{amsthm}
\usepackage{breqn}
\usepackage{mathtools}
\usepackage{amssymb}
\usepackage{enumitem}
\usepackage{xcolor}
\usepackage{hyperref}
\usepackage{dirtytalk}
\usepackage{ulem}%,refcheck}
\usepackage{caption,subcaption} % for subfigures

\usepackage{graphicx}
\usepackage{pgfplots}
\usetikzlibrary{math}

\usepackage{chngcntr}

\newcommand{\footremember}[2]{%
    \footnote{#2}
    \newcounter{#1}
    \setcounter{#1}{\value{footnote}}%
}
\newcommand{\footrecall}[1]{%
    \footnotemark[\value{#1}]%
} 

\counterwithin*{equation}{section}
\numberwithin{equation}{section}

\DeclareMathOperator{\sign}{sign}
\DeclareMathOperator{\supp}{supp}
\DeclareMathOperator{\meas}{meas}
\DeclareMathOperator{\loc}{loc}
\DeclareMathOperator{\totvar}{TV}
\DeclareMathOperator\Lip{Lip}

\newtheorem{theorem}{Theorem}[section]
\newtheorem{lemma}{Lemma}[section]

\newtheorem{proposition}[lemma]{Proposition}
\newtheorem{definition}[lemma]{Definition}

\theoremstyle{remark}
\newtheorem{remark}[lemma]{Remark}

\newcommand{\seq}[1]{{\left\{#1\right\}}_{N \in \mathbb{N}}}
\newcommand{\erho}{\rho^{E,N}}
\newcommand{\terho}{\tilde{\rho^{E,N}}}
\newcommand{\erhoM}{\rho^{E,M}}
\newcommand{\inverho}{X}

\newcommand{\erhotilde}{\tilde{\rho}^{E,N}}
\newcommand{\lrho}{\rho^{L,N}}

\newcommand{\drho}{\rho^{D,N}}

\newcommand{\norminfty}[2]{\left\lVert#1\right\rVert_{L^{\infty}(#2)}}
\newcommand{\normone}[2]{\left\lVert#1\right\rVert_{L^1(#2)}}

\newcommand{\normtwo}[2]{\left\lVert#1\right\rVert_{L^2(#2)}}

\newcommand{\deriv[1]}{\frac{d}{d#1}}

\newcommand{\sendlim}[1]{\lim_{N \rightarrow +\infty}{#1} = 0}
\newcommand{\sendlimk}[1]{\lim_{k \rightarrow +\infty}{#1} = 0}
\newcommand{\R}{\mathbb{R}}
\newcommand{\N}{\mathbb{N}}

\title
{On the continuum limit of the\\
Follow-the-Leader model and its stability}
\author{
Fabio Ancona\footremember{DM}{Dipartimento di Matematica ``Tullio Levi-Civita", Universit\`a di Padova, Italy. {\tt \small ancona@math.unipd.it, mohamed.bentaibi@math.unipd.it}} \footremember{Indam}{F. Ancona and F. Rossi are members of G.N.A.M.P.A. (I.N.d.A.M.).}%
  \and Mohamed Bentaibi\footrecall{DM}%
  \and Francesco Rossi\footnote{Dipartimento di Culture del Progetto, Università Iuav di Venezia, Italy
        {\tt\small francesco.rossi@iuav.it}} \footrecall{Indam}%
}

\pgfplotsset{compat=1.18}
\begin{document}

\maketitle
\begin{abstract}
    % The Follow-the-Leader model (FtL) is a dynamical system describing the motion of $N$ cars on a road lane, in which each car travels with a velocity that depends on its relative distance with respect to the one immediately in front. The Lighthill-Whitham-Richards (LWR) model is a hyperbolic conservation law where the solution is a macroscopic density that typically represents the dynamics of the average spatial concentration of vehicles. 
    
    % With the FtL model we build a microscopic density which approximates the macroscopic one. Our main goal is to prove that the microscopic density converges to the macroscopic one. This occurs under suitable hypotheses on the dynamics, as well as strong convergence requests on the initial data. Additional stability results of the FtL model are also presented.
    We consider the Follow-the-Leader (FtL) model and study which properties of the 
    initial positioning of the vehicles 
    ensure its convergence to the classical 
     Lighthill-Whitham-Richards (LWR) model 
     for traffic flow. Robustness properties of both FtL and LWR models with respect to the initial discretization schemes are investigated. Some numerical simulations are also discussed.
\end{abstract}
\tableofcontents

\section{Introduction}

Vehicular traffic on a one-lane  road can be described at
two fundamentally different levels:
microscopic and macroscopic.
The first one is based on the individual modeling of each vehicle, 
whose  dynamics is governed by the distance to the nearest vehicle  in front. 
This is the so-called Follow-the-Leader (FtL) model \cite{brachstoneMCD,gazis1961nonlinear},
which consists of a system of ordinary differential equations.
The other one, relying on a continuum assumption (better justified in the context of heavy traffic),  describes the traffic flow  in terms of an averaged density  that evolves according to a
partial differential equation. 
Assuming that the number of vehicles is conserved we get the classical 
Lighthill-Whitham-Richards (LWR) model~\cite{LW,richards}, an hyperbolic conservation law in which the averaged velocity is an explicit function  of the density.

The analysis of convergence of the microscopic FtL model  towards the macroscopic nonlinear conservation law LWR, as the number of vehicles tends to infinity and their length  tends to $0$, has been recently investigated by several authors (see~\cite{klarmatrascle,colrossi,fagioli2017,di2015rigorous,holden2017continuum,holdenrisebronum,rossielena} and references therein). The question can be summarized as follows, see also Figure \ref{f-problem}:
\begin{itemize}
    \item Take an initial macroscopic description, i.e. a probability density $\bar \rho$;
    \item Discretize it in a suitable manner, finding a microscopic description with $N+1$ vehicles with initial positions $$\bar x_0^N<
\bar x_1^N<\dots<\bar x_{N-1}^N<\bar x_{N}^N;$$
\item Let the microscopic model evolve in time via the FtL;
\item Compare it with the solution of LWR starting from $\bar\rho$.
\end{itemize}

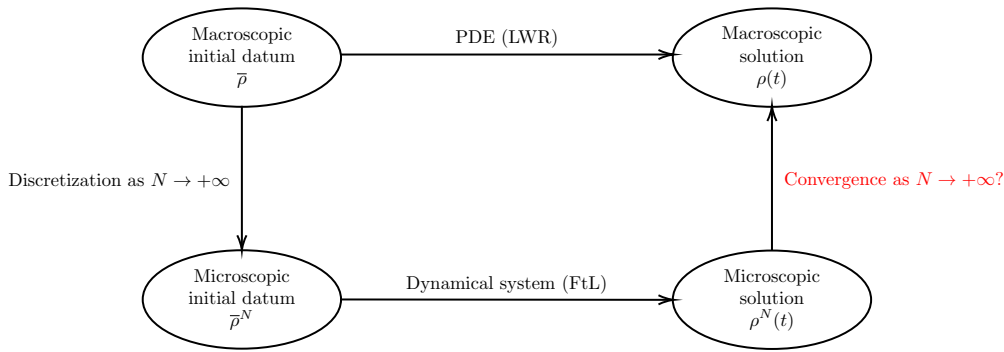
\begin{figure}[htb]
    \centering

\tikzset{every picture/.style={line width=0.75pt}} %set default line width to 0.75pt        

\begin{tikzpicture}[x=0.75pt,y=0.75pt,yscale=-0.7,xscale=0.7, trim left = -0.0cm]
%uncomment if require: \path (0,308); %set diagram left start at 0, and has height of 308

%Straight Lines [id:da7853180258264709] 
\draw    (197,240.79) -- (433,241.07) ;
\draw [shift={(435,241.07)}, rotate = 180.07] [color={rgb, 255:red, 0; green, 0; blue, 0 }  ][line width=0.75]    (10.93,-3.29) .. controls (6.95,-1.4) and (3.31,-0.3) .. (0,0) .. controls (3.31,0.3) and (6.95,1.4) .. (10.93,3.29)   ;
%Straight Lines [id:da7139694147633331] 
\draw    (127,203.64) -- (127,149) -- (127,101.65) ;
\draw [shift={(127,205.64)}, rotate = 270] [color={rgb, 255:red, 0; green, 0; blue, 0 }  ][line width=0.75]    (10.93,-3.29) .. controls (6.95,-1.4) and (3.31,-0.3) .. (0,0) .. controls (3.31,0.3) and (6.95,1.4) .. (10.93,3.29)   ;
%Straight Lines [id:da2504133636181274] 
\draw    (505,205.64) -- (505,103.65) ;
\draw [shift={(505,101.65)}, rotate = 90] [color={rgb, 255:red, 0; green, 0; blue, 0 }  ][line width=0.75]    (10.93,-3.29) .. controls (6.95,-1.4) and (3.31,-0.3) .. (0,0) .. controls (3.31,0.3) and (6.95,1.4) .. (10.93,3.29)   ;
%Straight Lines [id:da771512235855585] 
\draw    (197,64.79) -- (433,64.79) ;
\draw [shift={(435,64.79)}, rotate = 180] [color={rgb, 255:red, 0; green, 0; blue, 0 }  ][line width=0.75]    (10.93,-3.29) .. controls (6.95,-1.4) and (3.31,-0.3) .. (0,0) .. controls (3.31,0.3) and (6.95,1.4) .. (10.93,3.29)   ;

% Text Node
\draw    (127, 67) circle [x radius= 70.71, y radius= 35.36]   ;
\draw (127,67) node  [font=\footnotesize,xscale=0.7,yscale=0.7] [align=left] {\begin{minipage}[lt]{68pt}\setlength\topsep{0pt}
\begin{center}
Macroscopic initial datum \\$\displaystyle \overline{\rho }$
\end{center}

\end{minipage}};
% Text Node
\draw    (505, 67) circle [x radius= 70.71, y radius= 35.36]   ;
\draw (505,67) node  [font=\footnotesize,xscale=0.7,yscale=0.7] [align=left] {\begin{minipage}[lt]{68pt}\setlength\topsep{0pt}
\begin{center}
Macroscopic solution \\$ $$\displaystyle \rho ( t)$
\end{center}

\end{minipage}};
% Text Node
\draw    (127, 240.79) circle [x radius= 70.71, y radius= 35.36]   ;
\draw (127,240.79) node  [font=\footnotesize,xscale=0.7,yscale=0.7] [align=left] {\begin{minipage}[lt]{68pt}\setlength\topsep{0pt}
\begin{center}
Microscopic initial datum \\$\displaystyle \overline{\rho }^{N}$
\end{center}

\end{minipage}};
% Text Node
\draw    (505, 241) circle [x radius= 70.71, y radius= 35.36]   ;
\draw (505,241) node  [font=\footnotesize,xscale=0.7,yscale=0.7] [align=left] {\begin{minipage}[lt]{68pt}\setlength\topsep{0pt}
\begin{center}
Microscopic solution \\$ $$\displaystyle \rho ^{N}( t)$
\end{center}

\end{minipage}};
% Text Node
\draw (40.86,155.5) node  [font=\footnotesize,xscale=0.7,yscale=0.7] [align=left] {\begin{minipage}[lt]{120.47pt}\setlength\topsep{0pt}
Discretization as $N \rightarrow +\infty $
\end{minipage}};
% Text Node
\draw (594.05,155.5) node  [font=\footnotesize,color={rgb, 255:red, 255; green, 0; blue, 0 }  ,opacity=1 ,xscale=0.7,yscale=0.7] [align=left] {\begin{minipage}[lt]{120.51pt}\setlength\topsep{0pt}
Convergence as $N \rightarrow +\infty$?
\end{minipage}};
% Text Node
\draw (316,52.93) node  [font=\footnotesize,xscale=0.7,yscale=0.7] [align=left] {\begin{minipage}[lt]{60.86pt}\setlength\topsep{0pt}
\begin{center}
 PDE (LWR)
\end{center}

\end{minipage}};
% Text Node
\draw (316,229.86) node  [font=\footnotesize,xscale=0.7,yscale=0.7] [align=left] {\begin{minipage}[lt]{110.16pt}\setlength\topsep{0pt}
\begin{center}
Dynamical system (FtL)
\end{center}

\end{minipage}};

\end{tikzpicture}

\caption{Problem statement}
\label{f-problem}

\end{figure}

The first rigorous proof of convergence of this large particle limit was established in~\cite{di2015rigorous}, in which a very natural but specific discretization is proposed. The theory is based on a form of $L^1$ convergence of a suitable miscroscopic-like density to $\bar \rho$. Our article, relying on this first result, provides a more general answer, by showing that other discretization schemes can be chosen. Roughly speaking, we show that some form of weak convergence of the microscopic-like density is sufficient.\\

We now formally describe the framework of our contribution. Consider an initial probability density $\bar\rho$ 
with compact support, that satisfies $\|\bar\rho\|_{L^\infty}\leq\rho_{max}:= 1$.
%less or equal than $1$. 
Fix $N\in\mathbb{N}$ and choose $$\overline x_{min}:= \bar x_0^N<\bar x_1^N<\dots<\bar x_{N-1}^N<\bar x_{N}^N=: \overline x_{max},$$ that can be interpreted as the initial positions of $N+1$ ordered vehicles with mass $l:=\frac1N$. Let them evolve according to the ODE
\begin{equation}
\label{FtL-0}
    \dot{x}^N_i = v\left(\frac{l}{x^N_{i+1} - x^N_i}\right),\qquad i=0,\dots,N-1,
\end{equation}
where the velocity function $v=v(\rho)$  satisfies the standing assumptions:
\begin{align}\tag{V1}\label{e-V1}
    v \in \Lip([0, \rho_{\max}])\  
    \ \text{with Lipschitz constant}\ L, \qquad v(\rho_{\max})=0, \qquad v'(\rho)\leq c<0 \ \ \ \text{for a.e.} \ \rho.
    %\ \ \text{is decreasing}.
\end{align}
In some cases, an additional assumption will be required:
\begin{align}\tag{V2}\label{ass:furtheronv}
    \text{the map } [0,+\infty)\ni \rho \mapsto \rho\, v'(\rho) \in [0,+\infty) \text{ is non-increasing.}
\end{align}
Together with~\eqref{e-V1}, it implies the strict concavity of the map 
\begin{equation}
\label{eq:flux}
    \rho\mapsto f(\rho):=\rho\, v(\rho)\,.
\end{equation}
To close the system of ODEs~\eqref{FtL-0}, we prescribe the velocity of the first (leading) vehicle as the maximum possible 
velocity:
\begin{equation}
\label{FtL-0N}
    \dot{x}^N_N=v_{\max}:= v(0)\,.
\end{equation} 
One can view the quantity $l/(x^N_{i+1} - x^N_i)$ in~\eqref{FtL-0} as a discrete density, and consider as zero the value of the discrete density 
on the right of~$x_N^N$, since there is no other vehicle ahead of it.
Next, letting $x_i^N(t)$, $i=0,\dots,N$, denote the corresponding solutions of~\eqref{FtL-0}-\eqref{FtL-0N} with initial positions  $\bar x_i^N$,
one can define the discretized Eulerian density as
\begin{align}
\label{constructionofrhoN}
    \rho^{E,N}(t,x) \coloneqq \sum_{j=0}^{N-1} \frac{l}{x^N_{j+1}(t) - x^N_j(t)}\, \chi_{[x^N_j(t), x^N_{j+1}(t))}(x) \qquad x \in \mathbb{R},
    %\Omega.
\end{align}
where $\chi_A$ is the indicator function of a set $A$. Then, it is shown in~\cite{di2015rigorous} that, for a precise discretization scheme (that we recall in \eqref{eq:DFR-def} below), one has convergence in $L_\loc^1([0,+\infty) \times \mathbb{R}; [0,1])$ of  $\seq{\erho(t,x)}$ to the weak entropy solution $\rho(t,x)$ of the Cauchy problem for the LWR model 
\begin{equation}
    \label{lwr}
\begin{cases}
    \rho_{t}+f(\rho)_{x}=0, \quad\ t>0, \quad x \in \mathbb{R} \\
    \noalign{\smallskip}
\rho(0, x)=\bar{\rho}(x) \qquad x \in \mathbb{R},
\end{cases}
\end{equation}
with the flux $f(\rho)$ as in~\eqref{eq:flux}.
Notice that, by construction, here the 
 initial discretized density $\seq{\erho(0)}$ converges in $L^1(\mathbb{R})$ to the initial density $\bar\rho$.
A similar result was obtained in~\cite{holdenrisebronum,holdenrisebrohoftl}
for traffic density uniformly away from vacuum,
assuming the $L^1$ convergence of the inverse
Lagrangian discrete density (see Section \ref{s:lit}).

As explained above, in this paper we address the following 
questions: 
\begin{itemize}
[leftmargin=18pt]
    \item
which properties of the initial positioning of the vehicles and of the convergence of the discretized initial data ensure the convergence of the microscopic density $\rho^{E,N}$ to the macroscopic one $\rho$ as $N\to\infty$ ?
\item which kind of stability is enjoyed by these discretization schemes? 
\end{itemize}
An answer to these questions sheds light on the range of applicability, on the accuracy and on the robustness (with respect to errors, 
gaps in data collection and oscillations) of the many particle limit in the context of traffic flow. Moreover, from then modelling point of view, the analysis
of the discrete-to-continuum limit provides the theoretical background to reconstruct the traffic state of a region 
through  data collected from stationary detectors and GPS devices.
On the other hand, these results can be applied to 
validate the adoption of macroscopic LWR model 
in cases where the use of microscopic dynamics is better justified than the macroscopic one. 

Our results are all formulated for initial discretization schemes that have uniformly bounded support. Namely, we shall require that
the initial positions $x_i^N(0)=\bar x_i^N$ 
of all vehicles 
are contained in a fixed bounded set. 
\begin{definition}[Condition of uniformly bounded initial support]
    We say that $\{x_j^N(t)\}_{j=0}^N$ satisfies the condition of uniformly bounded initial support  if there exists a bounded set $K$
    %constant $K_1>0$ %independent of $N$, 
    such that  there holds
\begin{align}\label{cond:ubis}
%    x_N^N(0) - x_0^N(0) < K_1
     x_i^N(0)\in K
    \qquad\forall~i=0,\dots, N,\quad \forall~N \in \mathbb{N}.
\end{align}
\end{definition}
%

% ,
% where
% $\overline x_{max}, \overline x_{min}$,
% denote the extremal points of the convex hull of the support of $
% \overline\rho$.

The first main result of this paper basically shows that
we can replace the requirement of $L^1$
convergence of the initial discretization
(present both in~\cite{di2015rigorous} and in~\cite{holdenrisebronum,holdenrisebrohoftl})
with weak convergence.

\begin{theorem}\label{thm:micromacrointro}
Assume that the velocity map $v$ satisfies \eqref{e-V1}. 
    Let  $\bar{\rho} \in L^\infty(\mathbb{R}; [0,1])$ be with compact support and such that $\normone{\bar{\rho}}{\mathbb{R}}=1$.  Let $\{x_j^N(t)\}_{j=0}^N$ be solutions of the FtL system~\eqref{FtL-0}, \eqref{FtL-0N}, that moreover satisfy the condition of uniformly bounded initial support \eqref{cond:ubis}. Consider the corresponding Eulerian discrete density $\erho \in L^{\infty}\big([0,+\infty) \times \mathbb{R}; [0,1]\big)$ defined by \eqref{constructionofrhoN}. Assume that
     \begin{align}\label{ass:firsthypo}
        \erho(0)  \rightharpoonup \bar{\rho}\qquad\text{weak\,}^*\ \ \text{in}\ \ L^\infty(\mathbb{R}),
    \end{align}
and that one of the two following conditions hold:
\begin{enumerate}
    \item[(H1)]\label{thm:case1} $\bar{\rho}\in BV(\mathbb{R})$ 
    %is of bounded variation 
    and there exists $C > 0$ %independent of $N$ 
    such that $\totvar(\erho(0); \mathbb{R}) < C$ 
%    $\totvar(\erho(0);\Omega) < K$
for all $N$, i.e. such that
    \begin{align}\label{hyp: initial discrete TV bounded}
        %\frac{1}{N}
        \left(\frac{1}{x^N_1(0) - x^N_0(0)} + \frac{1}{x^N_N(0) - x^N_{N-1}(0)} + \sum_{j=0}^{N-2}\left|\frac{1}{x^N_{j+2}(0) - x^N_{j+1}(0)} - \frac{1}{x^N_{j+1}(0) - x^N_j(0)}\right| \right) < N\,C,
    \end{align}
    for all $N$;
    %\in \mathbb{N}$, or 
    
    \item[(H2)] \label{thm:case2} the velocity function $v$ satisfies~\eqref{ass:furtheronv}.
\end{enumerate}
Then the sequence $\seq{\erho}$ converges 
%strongly 
in $L_\loc^1([0,+\infty) \times \mathbb{R}; [0,1])$ to the weak entropy solution $\rho$ of the Cauchy problem \eqref{lwr}.
%as $N \rightarrow \infty$. 
\end{theorem}

    \begin{remark}
    % Another novelty here is t
    The proof 
    of convergence of the sequence of 
    Eulerian discrete density $\seq{\erho}$\,
    % of this paper is the proof of Theorem \ref{thm:micromacrointro}, 
    is based on an estimate of the $L^1$ Cauchy property of
    $\seq{\erho}$ in terms of the 
    $L^1$ Cauchy property of the
    cumulative distribution associated to $\rho^{E,N}$.
    Then one can conclude relying only
    % $\{F_{\rho^{E,N}}\}_{N\in \mathbb{N}}$
    % relies only 
    on the
    convergence of the cumulative and pseudoinverse functions associated to $\rho^{E,N}$, and on the
    $1$-Wasserstein convergence of $\seq{\erho}$
    that were established in~\cite{di2015rigorous}.
    This proof is
    %considerably 
    simpler than the one presented in \cite[Theorem 3]{di2015rigorous}, where the authors achieve the $L^1$-compactness of $\seq{\erho}$ 
    taking advantage also of the  Wasserstein  equicontinuity of 
    $\rho^{E,N}(t)$, which allows to 
    apply a generalization of the Aubin-Lions lemma. 
\end{remark}

The second main contribution of this paper is
 a stability result with respect to the 1-Wasserstein distance $W_1$.  It is a microscopic stability result for the evolution of two different initial discretization schemes, which in turn 
 yields a stability result 
 with respect to the $L^1$ norm that is uniform in time. 
Such a result is rather surprising in view of the 
instability of the FtL dynamics.

\begin{theorem}[Discrete Eulerian Stability Theorem]\label{thm:stabilitytheoremintro}
     Assume that the velocity map $v$ satifies \eqref{e-V1}. Let $\{x_j^N(t)\}_{j=0}^N$,$\{\tilde{x}_j^N(t)\}_{j=0}^N$ be solutions of the FtL system \eqref{FtL-0}-\eqref{FtL-0N}, 
     that moreover satisfy the condition of uniformly bounded initial support \eqref{cond:ubis}. 
     % Assume that $x_N^N(0) = \tilde{x}_N^N(0)$ for all $N \in \mathbb{N}$. 
     Consider the corresponding Eulerian discrete densities $\erho$, $\tilde{\rho}^{E,N} \in L^{\infty}\big(([0,+\infty) \times \mathbb{R});[0,1]\big)$ defined by \eqref{constructionofrhoN}.  
     Then, for all $T>0$, and for all $N \in \mathbb{N}$, there holds
\begin{equation}
\label{eq:wasse-stab}
\begin{aligned}
    \sup_{t \in [0,T]}W_1(\rho^{E,N}(t), \tilde{\rho}^{E,N}(t)) &\leq W_1(\rho^{E,N}(0), \tilde{\rho}^{E,N}(0)) + \\
    &\ \ +2LT\sum_{j=0}^{N-1}|x_{j+1}(0) - x_j(0) - (\tilde{x}_{j+1}(0) - \tilde{x}_j(0)) |,
\end{aligned}
\end{equation}  
where $L$ is the Lipschitz constant of $v$.
%      If there holds
%     \begin{align}\label{stabilityhypo}
%          \sendlim{\sum_{j=0}^{N-1}|x_{j+1}(0) - x_j(0) - (\tilde{x}_{j+1}(0) - \tilde{x}_j(0)) |},
%     \end{align}
% then we have
% \begin{align}
% \label{eq:Wlim-rho-trho}
%  \lim_{N \rightarrow + \infty}\sup_{t \in [0,T]}W_1(\rho^{E,N}(t), \tilde{\rho}^{E,N}(t)) = 0\qquad\forall~T>0.
% \end{align}
% {\color{red}Can we rephrase~\eqref{eq:Wlim-rho-trho} as
% \begin{equation}
% \label{eq:Wlim-rho-trho2}
%  \lim_{N \rightarrow + \infty}W_1(\rho^{E,N}(t), \tilde{\rho}^{E,N}(t)) = 0
%  \qquad \forall~t>0\ ?.
% \end{equation}}
Moreover, 
if there holds $x_N^N(0) = \tilde{x}_N^N(0)$ for all $N \in \mathbb{N}$, and
    \begin{align}\label{stabilityhypo}
         \sendlim{\sum_{j=0}^{N-1}|x_{j+1}(0) - x_j(0) - (\tilde{x}_{j+1}(0) - \tilde{x}_j(0)) |},
    \end{align}
then the following two properties are satisfied:
\begin{itemize}
  \item[(i)] if there exists $C > 0$ %independent of $N$ 
    such that
    $\totvar \left(\erho(0); \mathbb{R}\right),\totvar \left(\erhotilde(0); \mathbb{R}\right) < C$ for all $N$, then for all $T > 0$ there holds
\begin{align}
\label{unifconv-bvbdd}
      \sendlim{\sup_{\,t \in [0,T]}\normone{\rho^{E,N}(t) - \tilde{\rho}^{E,N}(t)}{\mathbb{R}}};
\end{align}

      \item[(ii)] if the velocity $v$ satisfies \eqref{ass:furtheronv}, then for all $T > 0$ 
      %$\delta,T > 0$ 
      there holds
\begin{align}
\label{unifconv-bvinf}
      \sendlimk{\sup_{\, t \in [1/k,\, T]}\normone{\rho^{E,N_k}(t) - \tilde{\rho}^{E,N_k}(t)}{\mathbb{R}}},
\end{align}
for some subsequences $\{{\rho}^{E,N_k}\}_k$\,, $\{\tilde{\rho}^{E,N_k}\}_k$\,.
\end{itemize}

\end{theorem}

%{\color{red} 
\begin{remark}
If $x_N^N(0) = \tilde{x}_N^N(0)$ for all $N$ and there holds~\eqref{stabilityhypo},
then Proposition~\ref{lagrangianimplieswasserstein} below ensures 
% $W_1(\rho^{E,N}(t), \tilde{\rho}^{E,N}(t))\to 0$ as $N\to +\infty$.
% Hence, applying~\eqref{eq:wasse-stab},
% one deduces 
% \begin{equation}
%     \lim_{N\to +\infty} W_1(\rho^{E,N}(t), \tilde{\rho}^{E,N}(t))=0
%     \qquad\forall~t>0\,.
% \end{equation}
\begin{equation*}
%\label{eq:wasselim-o}
     \lim_{N\to +\infty}W_1(\rho^{E,N}(0), \tilde{\rho}^{E,N}(0))=0\,.
\end{equation*}
This in turn implies that $\rho^{E,N}(0)-\tilde{\rho}^{E,N}(0)\rightharpoonup 0$.
Thus, letting $\bar\rho, \tilde \rho$
denote the weak* limit of $\{\rho^{E,N}(0)\}_N$, \,
$\{\tilde{\rho}^{E,N}(0)\}_N$,
respectively, we have 
$\bar\rho=\tilde \rho$.
Hence, applying Theorem~\ref{thm:micromacrointro}
we deduce  that both sequences
$\seq{\erho}$, $\seq{\terho}$, 
converge in $L_\loc^1([0,+\infty) \times \mathbb{R})$ to the weak entropy solution  of the Cauchy problem \eqref{lwr}, which implies
\begin{equation}
\label{eq:discr-densities-conv}
      \sendlim{\normone{\rho^{E,N}(t) - \tilde{\rho}^{E,N}(t)}{\mathbb{R}}}
      \qquad \text{for \ a.e.}\ t>0\,.
\end{equation}
The main new property provided by Theorem~\ref{thm:stabilitytheoremintro} is the fact that, thanks to the stability estimate~\eqref{eq:wasse-stab},
the convergence in~\eqref{eq:discr-densities-conv} is actually uniform in time.

Notice also that property (i) of Theorem~\ref{thm:stabilitytheoremintro}  implies that, if $x_N^N(0) = \tilde{x}_N^N(0)$ for all $N$, and if we have a uniform bound on the total variation
of $\rho^{E,N}(0)$, $\tilde{\rho}^{E,N}(0)$, then
the assumption~\eqref{stabilityhypo}
 in particular yields  the $L^1$ convergence $\rho^{E,N}(0)-\tilde{\rho}^{E,N}(0)\rightarrow 0$.
 %in $L^1(\mathbb{R})$.
\end{remark}

One final contribution of our work shows that a crucial question is still open. In Proposition~\ref{p:differentscheme-f}
    we give an example of a discretization scheme 
    that does not fulfill the assumption (H1)
    of Theorem~\ref{thm:micromacrointro}.
    Therefore, we cannot apply our result
    for such scheme
    in the case of fluxes $f(\rho)$ which are not concave.
    However, it is surprising to remark that the numerical simulations presented in Remark~\ref{rem:FtLnot convergingtoLWR} seem to suggest that the Eulerian discrete density defined with such a scheme
    exhibits essentially the same behavior of the
    one produced by the ones for which  Theorem~\ref{thm:micromacrointro} can be applied, ensuring convergence to solutions of LWR.
    This is an interesting phenomenon that shows that, in the case of non concave fluxes,  the relation between the convergence of  $\rho^{E,N}(0)$ to $\bar \rho$ and  of 
    $\rho^{E,N}(t)$ to the solution $\rho$ of~\eqref{lwr} has not yet been properly understood, and needs further investigation.\\

    The paper is organized as follows. In Section \ref{s:lit}, we compare our main results with the contributions of \cite{di2015rigorous} and \cite{holdenrisebronum,holdenrisebrohoftl}. In Section \ref{introftl} we recall the definition of the Follow-the-Leader dynamics and provide a stability result for it. In Section \ref{sectioneullagdefs} we define the Eulerian and Lagrangian discrete densities, their cumulative functions with the corresponding pseudo-inverses, and we discuss their properties and interpretations. In Section \ref{sectionmicromacro} we  prove the first main result of the article, i.e. Theorem \ref{thm:micromacrointro}. We also discuss 
in this section
an atomization scheme different from the ones in~\cite{di2015rigorous,holdenrisebronum,holdenrisebrohoftl}, which leads to an Eulerian discrete density that converges to the solution of the LWR model when the velocity $v$ satisfy the additional assumption (V2).
A numerical simulation indicating that this is not the case for  velocity $v$ that do not satisfy the assumption (V2) is also discussed in this section.
Finally, in Section~\ref{sectionmainstability} we establish the main stability result, i.e. Theorem \ref{thm:stabilitytheoremintro}.

%\noindent
% %{\color{red} 
% Also, can we provide an example showing that the estimate in~\eqref{unifconv-bvinf} is sharp
% since an estimate as~\eqref{unifconv-bvbdd} cannot be established in case (2) when $\rho^{E,N}(0), \tilde{\rho}^{E,N}(0)$ have unbounded total variation ?

\subsection{Comparison with the literature}\label{s:lit}
In this section, we compare our results with the most relevant other contributions in the field.

The main reference here is clearly \cite{di2015rigorous}. The main result there is Theorem 3, that provides the same convergence result under the following explicit discretization scheme: given $\bar\rho$, define 
\begin{equation}\label{eq:DFR-def-infpoint}
x_0^N:=\inf(\supp(\bar\rho))    
\end{equation}
and recursively 
\begin{equation}\label{eq:DFR-def}
{x}^N_j(0) \coloneqq \sup \left\{x \in \mathbb{R}: \quad \int_{x^N_{j-1}(0)}^x \bar{\rho}(y) dy < \frac{1}{N}\right\}, \qquad j=1,...,N.
\end{equation}
This amounts to split the subgraph of $\bar\rho$ into $N$ adjoining intervals of mass $l=1/N$ and to choose $x_i^N$ to be the extremes of these intervals. Remark that $x_0^N=\inf(\supp(\bar\rho))$ and $x_N^N=\sup(\supp(\bar\rho))$, for any $N$, i.e. that all discretizations share the initial and final points. Moreover, we will show in Proposition~\ref{DiFra-Ros-scheme} that this discretization ensures $L^1$-convergence of the Eulerian discrete density $\rho^{E,N}(0,x)$ to $\bar \rho$.

In our contribution, instead, we only require weak convergence of the Eulerian discrete density $\rho^{E,N}(0,x)$ to $\bar \rho$. In particular, it may well happen that $x_0^N(0)\neq \inf(\supp(\bar\rho))$
and $x_N^N(0)\neq \sup(\supp(\bar\rho))$. Since weak convergence does not provide information on the position of the initial and final point of the discretization scheme, we are forced to add the condition of uniformly bounded initial support \ref{cond:ubis}.

A very similar result is obtained for dense traffic regions (i.e. away from the vacuum) in~\cite[Theorem 2.5]{holdenrisebronum}, 
    \cite[Theorem 4.1]{holdenrisebrohoftl},
    where instead it is assumed:
    the $L^1$ convergence 
    of the inverse Lagrangian discrete density
    $y^{L,N}(0)$
    (see Definition~\ref{def:invdiscrldens} below) as $N\to\infty$; that 
    $y^{L,N}(0)$ has uniformly (in $N$) bounded total variation; and that the discrete density $\rho^{E,N}(0)$ is uniformly bounded away from zero. Moreover, in the same non-vacuum setting, \cite[Lemma 3.1]{ holdenrisebrohoftl} provides an $L^1$ stability estimate for different discretization schemes. In our contribution, in Theorem~\ref{thm:micromacrointro}  we are essentially providing a result showing that weak convergence implies strong convergence, even for initial data possibly containing vacuum regions.  In Theorem \ref{thm:stabilitytheoremintro}, we provide  the stability of two  different discretization schemes $\erho$ and $\erhotilde$, by
    exploiting the fact that weak convergence 
    combined with a control of the total variation
    implies strong convergence.   
    Weak convergence here is ensured by condition~\eqref{stabilityhypo}, which is based on the discretization scheme only.
    Such an hypothesis is assumed for instance in \cite[(2.11)]{holdenrisebronum} in the case of initial data $\bar{\rho} \in BV(\mathbb{R})$ away from vacuum.
    
Finally, in \cite[Theorem 3.6]{elioradicistra1}, the authors provide  a Cauchy property and the rate of convergence of a Eulerian microscopic density for non-local conservation laws. Also in this case, the result holds with a specific discretization scheme of the initial data $\bar{\rho} \in L^1(\mathbb{R}) \cap L^{\infty}(\mathbb{R})$, that satisfies $\bar{\rho}>0$ and $\int_{\mathbb{R}}|x|\bar{\rho}(x)dx < \infty$. This is given in the form of a microscopic stability between $\erho$ and $\rho^{E,M}$, for $M,N \in \mathbb{N}$ large enough. The main idea is that the Eulerian microscopic density is a quasi-entropy solution of the conservation law. The generalization of our results to  non-local conservation laws seems interesting, since they are somehow more naturally connected to microscopic dynamics, e.g. via the mean-field limit. This is a future research topic that we aim to address.

\section{The Follow-the-Leader model}\label{introftl}

In this section, we introduce the Follow-the-Leader (FtL) model and study its behaviour. It is a classical model for road traffic on a one-lane road with no overtaking, see e.g. \cite{10.2307/167610,gazis1961nonlinear}. The goal here is to investigate its stability properties with respect to the initial data. We first define the dynamics of the positions of vehicles $x_j(t)$, then consider the associated discrete density $\rho_j(t)$, and finally introduce the inverse discrete density $y_j(t)$. For each of these quantities, we analize the dynamics and some useful properties.

We start by considering $N+1$ vehicles, of length $l$, with initial positions \begin{equation}
\label{eq:initial-vehicles}
    \bar{x}^N_0 < \dots < \bar{x}^N_{N}
\end{equation} 
satisfying 
\begin{equation}
\label{eq:minimun-initial-distance}
\bar{x}^N_{i+1}-\bar{x}^N_i\geq l,\qquad\qquad\mbox{with}\quad 
    l\coloneqq \frac{1}{N}.
\end{equation}
This standard condition ensures non overlapping of vehicles. 
% \textcolor{blue}{Questa condizione, scritta così, è restrittiva. Ad esempio, usando la discretizazzione DFR per un dato iniziale $\bar \rho$, può darsi che esista un $i$ t.c. $\bar{x}^N_{i+1}-\bar{x}^N_i < l$, in una regione di alta densità di $\bar \rho$.}

We now define the FtL dynamics.
\begin{definition}
\label{def:FtL}
The FtL dynamics is
\begin{equation}\label{ftl}
    \left\{
    \begin{aligned}
    &\dot{x}^N_{N} = v_{max}, \\
    &\dot{x}^N_j = v\left(\frac{l}{x^N_{j+1} - x^N_j}\right), \qquad \text{for }j=0,...,N-1,\\
    &x^N_j(0)=\bar{x}^N_j,\hspace{1.05in} \text{for }j=0,...,N,
    \end{aligned}
    \right.
\end{equation}
where the initial positions 
$\bar{x}^N_j$ satisfy conditions~\eqref{eq:initial-vehicles}-\eqref{eq:minimun-initial-distance}.
\end{definition}

The FtL model describes the evolution of each vehicle $x^N_j$, which adapts its speed with respect to the distance with the vehicle immediately in front $x^N_{j+1}$.
%As in~\cite{di2015rigorous}, we also 
We now introduce the corresponding definition of discrete density and of its dynamics.
\begin{definition}
Given $\{x^N_j(t)\}_{j=0}^{N}$ a solution of \eqref{ftl}, define the discrete density as
 \begin{align}\label{defofrhoj}
     \rho^N_j(t) \coloneqq \frac{l}{x^N_{j+1}(t) - x^N_j(t)} \qquad j = 0,...,N-1.
 \end{align}
\end{definition}
Because of~\eqref{ftl}, the discrete density satisfies the dynamics
\begin{equation}
\label{ftlrho}
\left\{
     \begin{aligned}
      & \dot{\rho}^N_{N-1} =  -N(\rho^N_{N-1})^2\left(v_{\max}- v(\rho^N_{N-1}) \right), \\
      & \dot{\rho}^N_j =  N(\rho^N_j)^2\left( v(\rho^N_j) - v(\rho^N_{j+1}) \right),\qquad &\text{for }j=0,...,N-2, \\
      & \rho^N_j(0) = \bar{\rho}^N_j, \qquad &\text{for }j=0,...,N-1,
     \end{aligned}
     \right.
 \end{equation}
 where the initial data is 
 \begin{align*}
     \bar{\rho}^N_j \coloneqq \frac{l}{\bar{x}^N_{j+1} - \bar{x}^N_j}, \qquad\text{for } j = 0,...,N-1.
 \end{align*}
We finally consider the inverse discrete density introduced in~\cite{holdenrisebronum}.
\begin{definition}\label{def:defofinvlocaldens}
Given $\{x^N_j(t)\}_{j=0}^{N}$ a solution of \eqref{ftl}, define the inverse discrete  density as
    \begin{align}\label{def:defofyj}
    y^N_j(t) \coloneqq \frac{x^N_{j+1}(t) - x^N_j(t)}{l}= \frac{1}{\rho^N_j(t)} \qquad j = 0,...,N-1.
    \end{align} 
 \end{definition}
 Because of~\eqref{ftl}, the inverse discrete density satisfies the dynamics 
 \begin{equation}\label{ftly}
 \left\{
        \begin{aligned}
      & \dot{y}^N_{N-1} =  N\left(v_{\max} - V({y^N_{N-1}}) \right), \\
      & \dot{y}^N_j =  N\left(V(y^N_{j+1}) - V(y^N_{j}) \right),\qquad &\text{for } j=0,...,N-2 \\
      & y^N_j(0) = \bar{y}^N_j \coloneqq \frac{\bar{x}^N_{j+1}(t) - \bar{x}^N_j(t)}{l}, \qquad &\text{for } j = 0,...,N-1,
     \end{aligned}
     \right.
 \end{equation}
where the velocity of the inverse discrete density is defined by 
 \begin{eqnarray*}
     V(y) \coloneqq v\left(\frac{1}{y}\right).
 \end{eqnarray*}
 Here, the first equation of~\eqref{ftly}  prescribes that the inverse discrete density of the leading vehicle evolves with the maximum velocity
 \begin{align}
 \label{eq:veliddleader}
    V\big(y_N^N\big)=v(0)=v_{\max},
  \end{align}
  which could be viewed as setting \say{$y_N^N=+\infty$},
  corresponding to have an empty road in front of the leader~$x_N^N$.
As a consequence of \eqref{e-V1}, the velocity of the inverse discrete density satisfies the conditions
\begin{equation*}%\tag{V1'}\label{e-V2}
    V \in \Lip([1, +\infty)) \  
    \ \text{with Lipschitz constant}\ L, \qquad V\left(1\right)=0,\qquad  V \ \text{is increasing}.
\end{equation*} 

 \begin{remark}[Discrete Minimum/Maximum Principle]
 The solution of the FtL model \eqref{ftl} and the corresponding discrete density \eqref{ftlrho} satisfy a discrete minimum/maximum principle. This is the microscopic version of the well-known maximum principle enjoyed by solutions to \eqref{lwr},see for example \cite[Theorem 6.2.7]{dafermos2000hyperbolic}. Indeed,  the following estimates hold:
\begin{equation}
\label{eq:discr-maxprinc}
    \begin{aligned}
&\min_{j=0,\ldots, N-1} (x^N_{j+1}(t)-x^N_{j}(t))~\geq \min_{j=0,\ldots, N-1} (\bar x^N_{j+1}-\bar x^N_{j})~\geq l;\\
&\max_{j=0,\ldots, N-1} (x^N_{j+1}(t)-x^N_{j}(t))\leq 
% \max_{j=0,\ldots, N-1} (\bar x^N_{j+1}-\bar x^N_{j})\leq 
\bar x^N_N-\bar x^N_0+ t\, v_{\max}\,. \quad 
% \\
% &\max_{j=0,\ldots N-1} \rho_j^N(t)\leq \max_{j=0,\ldots N-1} \bar \rho_j^N\leq 1.
%\label{mp}
\end{aligned}
\end{equation}%
% \textcolor{blue}{Non capisco perché dovrebbe valere il secondo statement, non l'avevo messa io, ed è anche falso (basta considerare il caso in qui $\max_{j=0,\ldots, N-1} (x^N_{j+1}(t)-x^N_{j}(t)) = x^N_{N}(t)-x^N_{N-1}(t)$). Poi, non mi sembra che sia usato in qualche conto.}
Thus, by virtue of\eqref{eq:minimun-initial-distance}-\eqref{defofrhoj}, we also
deduce \begin{equation}\label{eq:discr-maxprinc-density}
    \max_{j=0,\ldots N-1} \rho_j^N(t)\leq \max_{j=0,\ldots N-1} \bar \rho_j^N\leq 1\,.
\end{equation}
A proof of~\eqref{eq:discr-maxprinc} can be found in~\cite[Lemma 1]{di2015rigorous}.
% Notice that~\eqref{eq:minimun-initial-distance}, together with~\eqref{defofrhoj}, implies
% \begin{equation}
%     \rho^N_j(t) \leq 1, \qquad \forall~j = 0,...,N-1.
% \end{equation}
% %
Similarly, the solution of the discrete inverse density \eqref{ftly} satisfies a discrete minimum principle due to \eqref{eq:discr-maxprinc}. Indeed, it holds
\begin{equation*}
%\label{eq:min-p-y}
\min_{j=0,\ldots N-1} y_j^N(t)\geq \min_{j=0,\ldots N-1} \bar y_j^N\geq 1.
\end{equation*}
\end{remark}

 In the same spirit of~\cite[Lemma 2.3]{holdenrisebronum}, 
  we now prove a stability
  %contraction 
  estimate for two different solutions of \eqref{ftlrho}.
  
   \begin{proposition}\label{contractinvrho}
 Consider two solutions $\{x^N_j(t)\}_{j=0}^{N}$, $\{\tilde{x}^N_j(t)\}_{j=0}^{N}$ of \eqref{ftl}, 
 with initial positions 
$\{\bar{x}^N_j\}_{j=0}^{N}$\,,
$\{\tilde{\bar{x}}^N_j\}_{j=0}^{N}$\,,
respectively. Let 
 $\{\rho^N_j(t)\}_{j=0}^{N-1}$, $\{\tilde{\rho}^N_j(t)\}_{j=0}^{N-1}$ be
 the  corresponding discrete densities defined by \eqref{defofrhoj},
 and let $\{y^N_j(t)\}_{j=0}^{N-1}$, $\{\tilde{y}^N_j(t)\}_{j=0}^{N-1}$ 
 be the corresponding inverse discrete densities defined by \eqref{def:defofyj}. Then, 
 %for all $T>0$, 
 there holds
 \begin{equation}\label{stabilityresultinv}
     \sum_{j=0}^{N-1}|\rho^N_j(t) - \tilde{\rho}^N_j(t)|~\leq  \sum_{j=0}^{N-1}|y^N_j(0) - \tilde{y}^N_j(0)|
     \qquad\forall~t\geq0\,.
 \end{equation}
 \end{proposition}

\begin{proof}
Throughout the proof we drop the superscript $N$ for simplicity of notation. We  consider two solutions of \eqref{ftly} 
%evolving on different times 
parametrized by two different variables
$t$ and $\tau$, and  use
the Kruzkov's doubling of variables method to provide the contraction estimate for the inverse densities. We finally rely on the maximum principle for the discrete densities to conclude. With this aim, we define
 \begin{eqnarray*}
     &&V_j(t) \coloneqq V(y_j(t)), \qquad \tilde{V}_j(\tau) \coloneqq V(\tilde{y}_j(\tau)). 
 \end{eqnarray*}
 We then  notice that, for $j=0,...,N-2$ it holds
 \begin{align*}
     \deriv[t]|y_j(t) - \tilde{y}_j(\tau)| &= N \sign(y_j(t) - \tilde{y}_j(\tau))(V_{j+1}(t) - V_j(t)) \\
     \deriv[\tau]|y_j(t) - \tilde{y}_j(\tau)| &= N \sign(y_j(t) - \tilde{y}_j(\tau))(\tilde{V}_j(\tau) - \tilde{V}_{j+1}(\tau)).
 \end{align*}
 Therefore, we deduce that, for $j=0,...,N-2$, we have
 %\begin{equation}
 \begin{align}
     \left(\deriv[t] \right.&+ \left.\deriv[\tau]\right)|y_j(t) - \tilde{y}_j(\tau)|
     \nonumber\\
     &= N \sign (y_j(t) - \tilde{y}_j(\tau))[V_{j+1}(t) - V_j(t) - \tilde{V}_{j+1}(\tau) + \tilde{V}_j(\tau)] 
     \nonumber\\
     & = N \left[- \sign(y_j(t) - \tilde{y}_j(\tau)) (V_j(t) - \tilde{V}_j(\tau)) + \sign(y_{j+1}(t) - \tilde{y}_{j+1}(\tau)) (V_{j+1}(t) - \tilde{V}_{j+1}(\tau))\right.
     \nonumber\\
     &\qquad\left. + (V_{j+1}(t) - \tilde{V}_{j+1}(\tau))[\sign(y_j(t) - \tilde{y}_{j}(\tau)) - \sign(y_{j+1}(t) - \tilde{y}_{j+1}(\tau)) \right] 
     \nonumber\\
     & \leq N\left[- \sign(y_j(t) - \tilde{y}_j(\tau)) (V_j(t) - \tilde{V}_j(\tau)) + \sign(y_{j+1}(t) - \tilde{y}_{j+1}(\tau)) (V_{j+1}(t) - \tilde{V}_{j+1}(\tau)) \right].
     \label{eq:dtttauineq}
 \end{align}
 %\end{equation}
 The last inequality can be recovered as follows:
 \begin{itemize}
     \item If
     \begin{align}
     \label{eq:case1-dv-ineq}
     y_j(t) \geq \tilde{y}_j(\tau) \text{ and }  y_{j+1}(t) \leq \tilde{y}_{j+1}(\tau),
 \end{align}
 then one has
 \begin{align*}
     V_{j+1}(t) - \tilde{V}_{j+1}(\tau) \leq 0,\qquad\quad \sign(y_j(t) - \tilde{y}_{j}(\tau)) - \sign(y_{j+1}(t) - \tilde{y}_{j+1}(\tau))\geq 0.
 \end{align*}
 % and 
 % \begin{align*}
 %    \sign(y_j(t) - \tilde{y}_{j}(\tau)) - \sign(y_{j+1}(t) - \tilde{y}_{j+1}(\tau))\geq0.
 % \end{align*}
%
 \item
 If
 \begin{align}
  \label{eq:case2-dv-ineq}
     y_j(t) \leq \tilde{y}_j(\tau) \text{ and } y_{j+1}(t) \geq \tilde{y}_{j+1}(\tau),
 \end{align}
then one has
 \begin{align*}
     V_{j+1}(t) - \tilde{V}_{j+1}(\tau) \geq 0,\qquad\quad 
      \sign(y_j(t) - \tilde{y}_{j}(\tau)) - \sign(y_{j+1}(t) - \tilde{y}_{j+1}(\tau))\leq0.
 \end{align*}
 % and 
 % \begin{align*}
 %    \sign(y_j(t) - \tilde{y}_{j}(\tau)) - \sign(y_{j+1}(t) - \tilde{y}_{j+1}(\tau))\leq0.
 % \end{align*}
%
 \item Otherwise, if neither~\eqref{eq:case1-dv-ineq} nor~\eqref{eq:case2-dv-ineq} are satisfied, then one has
 \begin{align*}
      \sign(y_j(t) - \tilde{y}_{j}(\tau)) - \sign(y_{j+1}(t) - \tilde{y}_{j+1}(\tau)) =0.
 \end{align*}
 
 \end{itemize}
Summing up the inequalities in~\eqref{eq:dtttauineq}, we find \begin{align*}
      \sum_{j=0}^{N-2}\left(\deriv[t] + \deriv[\tau]\right)|y_j(t) - \tilde{y}_j(\tau)| &\leq N \sign(y_{N-1}(t) - \tilde{y}_{N-1}(\tau) )[V_{N-1}(t) - \tilde{V}_{N-1}(\tau)].
 \end{align*}
 On the other hand , for $j=N-1$, it holds 
\begin{align*}
    \left(\deriv[t] + \deriv[\tau]\right)&|y_{N-1}(t) - \tilde{y}_{N-1}(\tau)|\\
     &= N \sign (y_{N-1}(t) - \tilde{y}_{N-1}(\tau))[v_{\max} - V_{N-1}(t) - v_{\max} + \tilde{V}_{N-1}(\tau) ]\\
     &=N \sign (y_{N-1}(t) - \tilde{y}_{N-1}(\tau))[\tilde{V}_{N-1}(\tau)  - V_{N-1}(\tau) ].
\end{align*}
Therefore, we conclude that
\begin{equation}
     \sum_{j=0}^{N-1}\left(\deriv[t] + \deriv[\tau]\right)|y_j(t) - \tilde{y}_j(\tau)|~\leq 0. \label{e-1}
\end{equation}
Relying on~\eqref{e-1}, we can complete the proof with
 the same arguments of the proof of~\cite[Lemma 2.3]{holdenrisebronum}. Namely, multiplying \eqref{e-1} by a non-negative test function $\phi(t,\tau)$ with $\phi \in C^\infty_0((0, \infty) \times (0, \infty))$, and then integrating by parts, one obtains
 \begin{align}
 \label{eq:e-2}
    \int_0^\infty \int_0^\infty (\phi_t + \phi_\tau)\sum_{j=0}^{N-1}|y_j(t) - \tilde{y}_j(\tau)|dtd\tau \geq 0.
 \end{align}
Next, choose 
\begin{align*}
    \phi(t, \tau) = \psi\left(\frac{t+\tau}{2}\right)\eta_{\epsilon}(t-\tau),
\end{align*}
where $\psi \in C^\infty_0((0, \infty)\times(0, \infty))$ is a non-negative function, and $\eta_\epsilon$ is a standard mollifier converging to the Dirac delta at the origin as $\epsilon\to 0$. Then, plugging this test function in~\eqref{eq:e-2} and sending $\epsilon \rightarrow 0$ we get
\begin{align}
 \label{eq:e-3}
    \int_0^\infty \psi'(t) \sum_{j=0}^{N-1}|y_j(t) - \tilde{y}_j(t)|dt \geq 0
\end{align}
Now, taking
 $\psi$ in~\eqref{eq:e-3} to be a smooth approximation of the characteristic function of the interval $(t_1,t_2) \subset (0, t)$ we get
\begin{align}
 \label{eq:e-4}
    \sum_{j=0}^{N-1}|y_j(t_2) - \tilde{y}_j(t_2)| \leq \sum_{j=1}^{N-1}|y_j(t_1) - \tilde{y}_j(t_1)|.
\end{align}
Then, 
letting $t_1 \rightarrow 0$ and $t_2 \rightarrow t$ in~\eqref{eq:e-4}, we obtain 
\begin{align}\label{invstab}
    \sum_{j=0}^{N-1}|y_j(t) - \tilde{y}_j(t)| \leq \sum_{j=1}^{N-1}|y_j(0) - \tilde{y}_j(0)|.
\end{align}
Finally, by using \eqref{invstab} and the maximum principle \eqref{eq:discr-maxprinc-density}, 
we find
\begin{align*}
    \sum_{j=0}^{N-1}|\rho_j(t) - \tilde{\rho}_j(t)| = \sum_{j=0}^{N-1}\rho_j(t)\tilde{\rho}_j(t)\left|y_j(t) - \tilde{y}_j(t)\right|&\leq \sum_{j=0}^{N-1}\left|y_j(t) - \tilde{y}_j(t)\right|\\
    &\leq \sum_{j=0}^{N-1}\left|y_j(0) - \tilde{y}_j(0)\right|,
\end{align*}
thus establishing~\eqref{stabilityresultinv}.
\end{proof}

\begin{remark}
In \cite{holdenrisebronum,holdenrisebrohoftl} the authors establish the contractive estimate~\eqref{invstab} 
assuming a uniform bound on the inverse discrete density $y_j^N(0)$
and on the total variation
of the corresponding inverse Eulerian discrete density $y^{E,N}$
(see Definition~\ref{def:invdiscrdens} below).
The estimates in~\cite{holdenrisebronum} were obtained
in two settings: 
\begin{itemize}
    \item either they assume to have  infinitely many equally spaced vehicles in front of the leading one located at $x_N^N$, 
with a distance $M / N$
 between two consecutive ones, 
for some constant $M>1$,
%moving with constant velocity,

\item or they assume that  the location of the vehicles is periodic in an interval $[a,b]$, so that 
the distance between the vehicle located in $x_N^N$ and the one located at $x_1^N$ is $(b-x_N^N)+(x_1^N-a)$.
\end{itemize}
This corresponds to define
the inverse discrete density related to the leading vehicle as \begin{align*}
     y^N_N =
     \begin{cases}
      M &\text{in non-periodic case} \\
      N(b - x_N + x_1 - a) & \text{in periodic case}.
     \end{cases}
 \end{align*}
%for some constant $M > 1$. 
In the non-periodic setting this definition leads to prescribe the velocity 
 \begin{align*}
     V\big(y_N^N\big)= v\!\left(\frac{1}{M}\right)
 \end{align*}
for the inverse discrete density in front of the leader. 
Here, instead, we obtain the contractive estimate~\eqref{invstab}
by observing that $\dot{x}^N_{N} = v_{max}$ in ~\eqref{ftl} implies the first equation in ~\eqref{ftly} with $V$ given by~\eqref{eq:veliddleader}.
% \begin{align*}
%     V\big(y_N^N\big)=v(0)=v_{\max}.
%   \end{align*}
Therefore, Proposition~\ref{contractinvrho}  provides an extension of \cite[Lemma 2.3]{holdenrisebronum}, in the non-periodic setting, to the case 
``$M = +\infty$" corresponding to empty road ahead of the leader,
and removing any boundedness assumption on $y_j^N(0)$
and on the total variation of $y^{E,N}$.
% consider a periodic and a non-periodic case by defining $y_N^N$ ``in front of'' $y_{N-1}^N(t)$ as  
%  \begin{align*}
%      y^N_N =
%      \begin{cases}
%       M, &\text{non-periodic case} \\
%       N(b - x_N + x_1 - a) & \text{periodic case}
%      \end{cases}
%  \end{align*}
%  where $M > 1$. In the non-periodic case, they consider to have  infinitely many equally spaced vehicles moving with constant velocity 
%  \begin{align*}
%      V\left(y_N^N\right)= v\left(\frac{1}{M}\right)
%  \end{align*}
%  in front of the leader.  They then get the contraction estimate \eqref{invstab} for both non-periodic and periodic case.  The estimate \eqref{invstab} is a generalization of such a result in the sense that it represents the case 
%``$M = +\infty$". 
% Indeed, in our case we actually consider
%   \begin{align*}
%      V\left(y_N^N\right)=v(0)=v_{\max}.
%  \end{align*}

\end{remark}

Finally, we recall the discrete Oleinik-type condition proved in \cite[Corollary 1 of Lemma~6]{di2015rigorous}
in the case of the particular discretizazion scheme considered therein,
which remains valid for a general discretization scheme. 
Such one-sided estimate yields
 the uniform bounds on the total variation of 
the discrete densities stated in Proposition~\ref{prop: contract Tot Var}-(ii) below.

\begin{lemma}[Discrete Oleinik-type condition]\label{oleinik}
    Consider a solution  $\{x^N_j(t)\}_{j=0}^{N}$ of \eqref{ftl}, and let   $\{\rho^N_j(t)\}_{j=0}^{N-1}$ be
 the  corresponding discrete density defined by \eqref{defofrhoj}.    
    % $\{\rho^N_j(t)\}_{j=0}^{N-1}$ solution of system \eqref{ftlrho}. 
    Assume that $v$ satisfies \eqref{e-V1} and \eqref{ass:furtheronv}. Then, for any $j = 0, \ldots, N-2$, there holds
\begin{align*}
    % t\rho_j^N(t) \left[v\left(\rho_{j+1}^N(t)\right) - v\left(\rho_j^N(t)\right) \right] \leq \frac{1}{N} 
    \frac{v\big(\rho_{j+1}^N\big(t, x_{j+1}^N(t)\big)\big) - v\big(\rho_j^N\big(t, x_j^N(t)\big)\big)}{x_{j+1}^N(t)-x_j^N(t)}\leq\frac{1}{t}\qquad \forall \,t \geq 0.   
\end{align*}
\end{lemma}
\begin{proof}
    The proof in \cite[Corollary 1 of Lemma~6]{di2015rigorous} is completely independent of the initial profile $\{x^N_j(0)\}_{j=0}^{N}$, and on
    the initial discretization 
    $\{\rho^N_j(0)\}_{j=0}^{N}$. It depends solely on the dynamics given by \eqref{ftlrho}.
\end{proof}

\section{Eulerian and Lagrangian densities}\label{sectioneullagdefs}

In this section, we present several different densities that approximate the solution of \eqref{lwr}. They have a simple structure, being either piecewise constant or a combination of Dirac deltas.
This section is mainly based on the analysis developed in~\cite{di2015rigorous}. 

% first define two densities that approximate the solution of \eqref{lwr}: the Eulerian discrete density and the (Dirac) empirical measure. We then define the Lagrangian discrete density and the inverse Lagrangian discrete density,
% which provide an approximation
% of the solution of~\eqref{lwr} expressed in Lagrangian coordinates. After that, we define the cumulative function of the Eulerian density and its corresponding pseudo-inverse, which are used to 
% transform the Lagrangian density into the Eulerian one and viceversa.
%  In the end, we provide some convergence results. This section is mainly based on the analysis developed in~\cite{di2015rigorous}. \textcolor{blue}{A Francesco non piace questa intro, non so come modificarlo.}

We first introduce the Eulerian discrete density, that can be understood as a discrete approximation of the solution of the LWR model \eqref{lwr}, based on the dynamics of the FtL~\eqref{ftl}. 
%Its precise definition is given here. 
\begin{definition}
\label{def:eul-dens}
Given $\{x^N_j(t)\}_{j=0}^N$ solution of \eqref{ftl}, define the Eulerian discrete density as
\begin{align}\label{def:erho}
    \rho^{E,N}(t,x) \coloneqq \sum_{j=0}^{N-1} \rho^N_j(t) \chi_{[x^N_i(t), x^N_{i+1}(t))}(x),
\end{align}
where $\rho^N_j$ are defined by~\eqref{eq:minimun-initial-distance}-\eqref{defofrhoj}.
\end{definition}
Notice that the Eulerian discrete density can  be seen as a quasi-entropy solution of~\eqref{lwr}, as discussed in \cite{elioradicistra1}. We now define the inverse Eulerian discrete density and the (Dirac) empirical measure.
\begin{definition}
\label{def:invdiscrdens}
Given $\{x^N_j(t)\}_{j=0}^N$ solution of \eqref{ftl}, define the inverse Eulerian discrete density as
\begin{align*}
    y^{E,N}(t,x) \coloneqq \sum_{j=0}^{N-1} y^N_j(t) \chi_{[x^N_i(t), x^N_{i+1}(t))}(x),\qquad x\in\mathbb{R},
\end{align*}
and the (Dirac) empirical measure as
\begin{align}\label{def:drho}
    \rho^{D,N}(t,x) \coloneqq \frac{1}{N} \sum_{j=0}^{N-1} \delta_{x_j(t)}(x),
    \qquad x\in\mathbb{R},
\end{align}
where $y^N_j$ are defined by~\eqref{def:defofyj},
and $\delta_x$ denotes the Dirac delta at  point $x$.
\end{definition}

We finally define the Lagrangian discrete density and the inverse Lagrangian density. The latter can be understood as a piecewise constant approximation of the solution of the Lagrangian version of the LWR model, see \cite{holdenrisebronum}. We recall that $l=1/N$.
\begin{definition}
\label{def:invdiscrldens}
 Given $\{x^N_j(t)\}_{j=0}^N$ solution of \eqref{ftl}, define the Lagrangian discrete density as
\begin{align}\label{def:lrho}
   \rho^{L,N}(t,z) \coloneqq \sum_{j=0}^{N-1}\rho^N_j(t) \chi_{[jl,(j+1)l)}(z), 
   \qquad z\in[0,1],
\end{align}
and the inverse Lagrangian discrete density as
\begin{align}\label{def:ly}
     y^{L,N}(t,z) \coloneqq \sum_{j=0}^{N-1}y^N_j(t) \chi_{[jl,(j+1)l)}(z),
     \qquad z\in[0,1].
 \end{align}
\end{definition}

The coordinate $z \in [0,1]$ can be seen as a Lagrangian mass coordinate. As pointed out in \cite{holdenrisebronum}, the integer part of $\frac{z}{l}$ measures how many vehicles are located to the left of $z$. 

% Now that we have defined both $\erho$ and $y^{L,N}$, we can explain their connections. W
Notice that, while 
the $L^1$ norm 
%in Eulerian coordinates 
of the Eulerian discrete density $\erho$ represents the total mass of vehicles, the $L^1$ norm 
%in Lagrangian coordinates 
of the inverse Lagrangian discrete density $y^{L,N}$ provides the measure of their support.
%of $\erho$. 
Indeed, given $\{x^N_j(t)\}_{j=0}^N$ solution of \eqref{ftl}, it holds
\begin{align}
\label{eq:Ldensl1norm}
    \normone{y^{L,N}(t)}{[0,1]} = 
    \sum_{j=0}^{N-1} y_j^{N}(t) \cdot l
%    \int_0^1 \chi_{[jl,(j+1)l)}(z) dz 
    = \sum_{j=0}^{N-1}x^N_{j+1}(t) - x^N_j(t) = x^N_N(t) - x^N_0(t).
\end{align}
% \textcolor{blue}{Secondo me, forse è meglio lasciare il $\int_0^1 \chi_{[jl,(j+1)l)}(z) dz$ invece di scrivere direttamente $l$.}
Therefore, if $\{x^N_j(t)\}_{j=0}^N$ satisfy the condition of uniformly bounded initial support \eqref{cond:ubis},
relying on the discrete maximum principle~\eqref{eq:discr-maxprinc}
we deduce that
the corresponding inverse Lagrangian discrete density $y^{L,N}(t)$ 
has a bound in $L^1([0,1])$
% remains uniformly bounded with respect to $N$ %in $L^1([0,T] \times [0,1])$ 
% in $L^1([0,1])$, 
that is uniform with respect to~$N$, for all $t>0$.
 
% \textcolor{blue}{Francesco ha chiesto di scrivere \cite[Proposition 5]{di2015rigorous} in questa sezione. Non mi sembra che conviene scriverla nelle sezioni 3.1 e 3.2, quindi la metto qui e magari vediamo se c'è un posto migliore}

The discrete Eulerian and Lagrangian densities enjoy
 a BV
contraction property in the case of initial data with bounded variation,
and uniform BV estimates
%contraction result 
%on the total variation of 
for initial data with velocity satisfying assumption (V2).
These results are
established in~\cite[Propositions 5-6]{di2015rigorous}
(see also~\cite[Propositions 1-3]{fagioli2017}),
and collected in the next proposition.

\begin{proposition}
%[\cite{di2015rigorous}]
\label{prop: contract Tot Var}
Assume that $v$ satisfies (V1).
    Let $\{x_j^N(t)\}_{j=0}^N$ be a solution of \eqref{ftl}. Consider the corresponding Eulerian discrete density $\erho \in L^{\infty}([0,+\infty) \times \mathbb{R}; [0,1])$ defined by \eqref{def:erho} and the Lagrangian discrete density $\lrho \in L^{\infty}([0,+\infty) \times [0,1])$ defined by \eqref{def:lrho}. 
    Then, the following hold:
    \begin{itemize}
        \item[(i)]\label{prop item: totvar decreasing}
        if $\bar \rho \in BV(\mathbb{R})$, then
        % for any $N \in \mathbb{N}$ it holds
    \begin{align*}
        \totvar \left(\erho (t);\,\mathbb{R}\right) = \totvar \left(\lrho (t);\,\mathbb{R}\right) \leq \totvar \left(\erho (0);\,\mathbb{R}\right)
        %\totvar(\bar \rho) 
        \qquad \forall~t \geq 0,
        \quad\forall~N \in \mathbb{N};
            \end{align*}
            \item[(ii)]
        if the velocity function $v$ satisfies (V2) and $\bar \rho \in L^\infty(\mathbb{R})$,
        then, 
        % for any given bounded $\Omega \subset \mathbb{R}$ and 
        for any $\delta>0$ there exists a constant $C>0$,
        depending on 
        %$\meas(\Omega)$ and
        ~$\delta$, such that
        \begin{equation*}
            \sup_{t\geq\delta}
            \totvar \left(\erho (t);\,\mathbb{R}
            %\Omega
            \right)\leq C,
            \qquad
            \sup_{t\geq\delta}
            \totvar \left(\lrho (t);\,\mathbb{R}
            %\Omega
            \right)\leq C,
            \quad\ \forall~N \in \mathbb{N}.
        \end{equation*}
    \end{itemize}
\end{proposition}
% \MB{Ho aggiunto questa proof per mostrare che, come nel caso del Oleinik condition, la discretizazzione iniziale non c'entra, e che si considera solamente la dinamica del FtL.}
\begin{proof}
    % The proof of  statement (i) is the same as in \cite[Proposition 5]{di2015rigorous}. Indeed, b
    We prove statement (i) by proving that 
    \begin{align*}
        \frac{d}{dt} \totvar \left(\erho (t)\right) \leq 0,\qquad \forall~t>0,\quad \forall~N\in  \mathbb{N}\,.
    \end{align*}
    %following 
    It is sufficient to apply the exact same computations 
    of the proof of~\cite[Proposition 5]{di2015rigorous},
    which are indipendent on the particular
    initial discretization scheme $\{x^N_j(0)\}_{j=0}^N$.

    We prove statement (ii),
 by %    we proceed as follows. F
    following the exact same computations as in \cite[Proposition 6]{di2015rigorous}, that are as well indipendent on the particular
    initial discretization scheme $\{x^N_j(0)\}_{j=0}^N$. By  relying on Lemma~\ref{oleinik},
    we find
    that, for any $N \in \mathbb{N}$ and for all $t \geq \delta$, it holds
    \begin{align*}
            \totvar \left(v\left(\erho (t)\right);\,\mathbb{R}
            %\Omega
            \right)=    \totvar \left(v\left(\lrho (t)\right);\,\mathbb{R}
            %\Omega
            \right) &\leq \left[3 v_{\max} + 2 \frac{x^N_N(0) - x^N_0(0)}{\delta}\right]\\
            & \leq \left[3 v_{\max} + 2 \frac{\meas (K)}{\delta}\right].
        \end{align*}
    In the last inequality, $K$ is the bounded set such that \eqref{cond:ubis} holds. Since $v$ is invertible 
    and $(v^{-1})'$ is bounded 
    because of condition (V1), statement (ii) follows. 
    % \MB{Questa ultima frase è quella che hanno scritto Difra, Fagioli e Rosini, ma non sono convinto che sia abbastanza. Sia in Difra-Rosini che in Difra-Fagioli-Rosini, chiedono $v \in C^1([0,+\infty))$ e $v'(\rho) < 0$ per $\rho > 0$. Nelle ipotesi sulla velocità che chiediamo noi, penso che ci sia bisogno di precisare che $v'(\rho) \leq -c < 0$ per qualche $c>0$. In generale, assumendo che $\rho$ sia smooth, avremmo allora
    % \begin{align*}
    %     \totvar (\rho) = \int_{\mathbb{R}}|\rho_x|dx = \int_{\mathbb{R}}\left|\rho_x \frac{v'(\rho)}{v'(\rho)}\right| dx = \int_{\mathbb{R}}\left|\partial_x v(\rho)\frac{1}{v'(\rho)}\right|dx &= \int_{\mathbb{R}}\left|\partial_x v(\rho)\right|\frac{1}{|v'(\rho)|}dx  \\
    %     &\leq \frac{1}{c} \int_{\mathbb{R}}\left|\partial_x v(\rho)\right|dx = \frac{1}{c} TV(v(\rho)).
    % \end{align*}
    % Per $\rho$ non smooth si potrebbe magari usare lo stesso argomento per degli approsimazzioni mollified. Non mi è chiaro come si può dedurre che $TV(v(\erho))$ bounded implica che $TV(\erho)$ sia bounded usando solamente la strict monotonicity ($v' < 0$) e la continuità di $v$. 
    %      }

\end{proof}

\subsection{Cumulative and pseudo-inverse functions} We now define the cumulative distribution of a function and the corresponding pseudo-inverse.

\begin{definition}\label{defofFandX} Consider the space of probability densities
\begin{align*}
    \mathcal{P}_c(\mathbb{R}) \coloneqq \{\rho \text{ Radon measure on $\mathbb{R}$ with compact support with } \rho(\mathbb{R})=1\}.
\end{align*}

Given $\rho \in \mathcal{P}_c(\mathbb{R})$, define the cumulative distribution $F_{\rho}: \mathbb{R} \mapsto [0,1]$:
\begin{align}
\label{eq:cumul-def}
    F_{\rho}(x) \coloneqq \rho((-\infty, x]),
%    \int_{-\infty}^x\rho(y)dy, 
\qquad x\in \mathbb{R},
\end{align}
and its associated pseudo-inverse $X_{\rho}: [0,1] \mapsto \mathbb{R}$ as
\begin{equation*}%\label{Xrho}
  X_{\rho}(z) \coloneqq  \inf\{ x \in \mathbb{R} \quad | \quad F_{\rho}(x) \geq z\},\qquad z\in [0,1].
\end{equation*}
\end{definition}
Observe that $F_{\rho}$ is non-decreasing and right-continuous.

 We recall that the one dimensional Wasserstein distance can be defined using the cumulative or the pseudo-inverse functions, see e.g. \cite{villani2021topics}.
\begin{definition}\label{def:wassersteindistance}
The one-dimensional 1-Wasserstein distance is
\begin{align}\label{defofwass}
    W_1(\rho,\tilde{\rho}) \coloneqq \normone{F_{\rho} - F_{\tilde{\rho}}}{\mathbb{R}}= \normone{X_{\rho}-X_{\tilde{\rho}}}{[0,1]}. 
\end{align}
\end{definition}
Recall that the discrete density $\rho^{E,N}$
is a probability measure in $\mathcal{P}_c(\mathbb{R})$.
We can apply Definition \ref{defofFandX} to  $\rho^{E,N}$
and find that its cumulative distribution takes the form:
\begin{equation}
\label{Frho}
\begin{aligned}
    F_{\rho^{E,N}}(t,x) 
    &= 
    \int_{-\infty}^x
    \rho^{E,N}(t,y) dy
    \\
    &=\sum_{j=0}^{N-1} \left[jl + \rho_j^{N}(t) (x-x_j(t))\right] \chi_{[x_j(t),x_{j+1}(t))}(x) + \chi_{[x_N(t),+\infty)}(x).
\end{aligned}
\end{equation}
Notice that the cumulative
distribution $F_{\rho^{E,N}}$
is $1$-Lipschitz in the $x$-variable.
The corresponding pseudo-inverse takes the form:
\begin{align}\label{pseudoinverseeul}
    X_{\rho^{E,N}}(t,z)=\sum_{j=0}^{N-1}\left[x^N_j(t) + \frac{z-jl}{\rho_j^N(t)}\right]\chi_{[jl,(j+1)l)}(z) + \left[x^N_{N}(t)\right]\chi_{\{1\}}(z),\quad z\in[0,1].
\end{align}
The pseudo-inverse $X_{\rho^{E,N}}$ satisfies
\begin{equation*}%\label{pushfwd lag to eul}
   \lrho(t,z) = \erho(t, \inverho_{\erho}(t,z)),  \quad  y^{L,N}(t,z) = y^{E,N}(t, \inverho_{\erho}(t,z))
   \qquad \forall~t\geq 0, \ z\in [0,1].
\end{equation*}
The cumulative function $F_{\rho^{E,N}}$ then satisfies
\begin{align}\label{pushfwd eul to lag}
   \lrho(t, F_{\erho}(t,x)) = \erho(t,x), \quad  y^{L,N}(t, F_{\erho}(t,x)) = y^{E,N}(t,x)
   \qquad \forall~t\geq 0, \ x\in \mathbb{R}.
\end{align}
Relying on~\eqref{pushfwd eul to lag}, one deduces that
\begin{equation*}
%\label{eq:l1lagrL2eul}
    \normone{\lrho(t)}{[0,1]}=\int_{\R}
     \erho(t,x)\,\frac{d}{dx} F_{\erho}(t,x) \,dx=\left\|\rho^{E,N}(t)\right\|^2_{L^2(\R)}\,.
\end{equation*}
Similarly, if we apply Definition \ref{defofFandX} to the (Dirac) empirical measure $\rho^{D,N}$, the cumulative distribution takes the form
\begin{align}\label{FDrho}
    F_{\rho^{D,N}}(t,x) = \sum_{j=0}^{N-1} \left[(j+1)l\right] \chi_{[x_j(t),x_{j+1}(t))}(x) + \chi_{[x_N(t),+\infty)}(x).
\end{align}
The corresponding pseudo-inverse takes the form:
\begin{align}\label{pseudoinversedir}
    X_{\rho^{D,N}}(t,z)=\sum_{j=0}^{N-1}\left[x^N_j(t) \right]\chi_{[jl,(j+1)l)}(z) + \left[x^N_N(t) \right]\chi_{\{1\}}(z).
    \end{align}

\begin{figure}
    \centering
\begin{tikzpicture}[scale=0.7]
% Definitions
\tikzmath{
\r1 =1;
\r2=0.65;
\r3=2;
\r4=0.35;
\x1=0;
\x2=\x1 + 1/4*\r1;
\x3=\x2 + 1/4*\r2;
\x4=\x3 + 1/4*\r3;
\x5=\x4 + 1/4*\r4;
 } 
% Axis
\begin{axis}[
axis x line=middle,
axis y line=middle,
xtick={\x1,\x2,\x3,\x4,\x5},
xticklabels={0, $x_1^4(t)$, $x_2^4(t)$, $x_3^4(t)$, $x_4^4(t)$},
extra x ticks={0},                                     
extra x tick style={xticklabel=$x_0^4(t)$, style={anchor=north east}},
xlabel near ticks,
ytick={0, \r1, \r2, \r3, \r4},
yticklabels={0,$\rho^4_0(t)$, $\rho^4_1(t)$, $\rho^4_2(t)$, $\rho^4_3(t)$},
ylabel near ticks,
tick label style={font=\tiny} ,
xmax=1,
ymax=2,
xmin=0,
ymin=0
]
% Plots
\addplot[domain=\x1:\x2] {\r1};
\addplot[domain=\x2:\x3] {\r2};
\addplot[domain=\x3:\x4] {\r3};
\addplot[domain=\x4:\x5] {\r4};
\draw[dotted] (axis cs:\x2,0) -- (axis cs:\x2, \r1);
\draw[dotted] (axis cs:\x3,0) -- (axis cs:\x3, \r3);
\draw[dotted] (axis cs:\x4,0) -- (axis cs:\x4, \r3);
\draw[dotted] (axis cs:\x5,0) -- (axis cs:\x5, \r4);
\addplot[only marks,mark=*] coordinates{(0,\r1)(\x2,\r2)(\x3,\r3)(\x4,\r4)};
\addplot[fill=white,only marks,mark=*] coordinates{(\x2,\r1)(\x3,\r2)(\x4,\r3)};
\end{axis}
\node[above] at (current bounding box.north) {$\rho^{E,N}(t,x)$};
\end{tikzpicture}
\begin{tikzpicture}[scale=0.7]
% Definitions
\tikzmath{
\r1 =1/1;
\r2=1/0.65;
\r3=1/2;
\r4=1/0.35;
 } 
% Axis
\begin{axis}[
axis x line=middle,
axis y line=middle,
xtick={0,1/4,1/2,3/4,1},
xticklabels={0,$\frac{1}{4}$, $\frac{1}{2}$, $\frac{3}{4}$, $1$},
extra x ticks={0},                                     
extra x tick style={xticklabel=0, style={anchor=north east}},
xlabel near ticks,
ytick={0, \r1, \r2, \r3, \r4},
yticklabels={0,$y^4_0(t)$, $y^4_1(t)$, $y^4_2(t)$, $y^4_3(t)$},
ylabel near ticks,
tick label style={font=\tiny} ,
xmax=1,
ymax=1/0.34,
xmin=0,
ymin=0
]
% Plots
\addplot[domain=0:1/4] {\r1};
\addplot[domain=1/4:1/2] {\r2};
\addplot[domain=1/2:3/4] {\r3};
\addplot[domain=3/4:1] {\r4};
\draw[dotted] (axis cs:1/4,0) -- (axis cs:1/4, \r2);
\draw[dotted] (axis cs:1/2,0) -- (axis cs:1/2, \r2);
\draw[dotted] (axis cs:3/4,0) -- (axis cs:3/4, \r4);
\draw[dotted] (axis cs:1,0) -- (axis cs:1, \r4);
\addplot[only marks,mark=*] coordinates{(0,\r1)(1/4,\r2)(1/2,\r3)(3/4,\r4)};
\addplot[fill=white,only marks,mark=*] coordinates{(1/4,\r1)(1/2,\r2)(3/4,\r3)};
\end{axis}
\node[above] at (current bounding box.north) {$y^{L,N}(t,z)$};
\end{tikzpicture}
\begin{tikzpicture}[scale=0.7]
% Definitions
\tikzmath{
\r1 =1;
\r2=0.65;
\r3=2;
\r4=0.35;
\x1=0;
\x2=\x1 + 1/4*\r1;
\x3=\x2 + 1/4*\r2;
\x4=\x3 + 1/4*\r3;
\x5=\x4 + 1/4*\r4;
 } 
% Axis
\begin{axis}[
axis x line=middle,
axis y line=middle,
yticklabel=\empty,
xtick={\x1,\x2,\x3,\x4,\x5},
xticklabels={0, $x_1^4(t)$, $x_2^4(t)$, $x_3^4(t)$, $x_4^4(t)$x},
extra x ticks={0},                                     
extra x tick style={xticklabel=$x_0^4(t)$, style={anchor=north east}},
xlabel near ticks,
tick label style={font=\tiny} ,
xmax=1,
ymax=2,
xmin=0,
ymin=0
]
% Plots
\draw (axis cs:\x2,0) -- (axis cs:\x2, 2);
\draw (axis cs:\x3,0) -- (axis cs:\x3, 2);
\draw (axis cs:\x4,0) -- (axis cs:\x4, 2);
\addplot[only marks,mark=triangle] coordinates{(0,2)(\x2,2)(\x3,2)(\x4,2)};
\end{axis}
\node[above,font=\bfseries] at (current bounding box.north) {$\rho^{D,N}(t,x)$};
\end{tikzpicture}
    \caption{The Eulerian discrete density, the inverse Lagrangian discrete density and the (Dirac) empirical measure profiles ($N=4$).}
    \label{fig:my_label}
\end{figure}
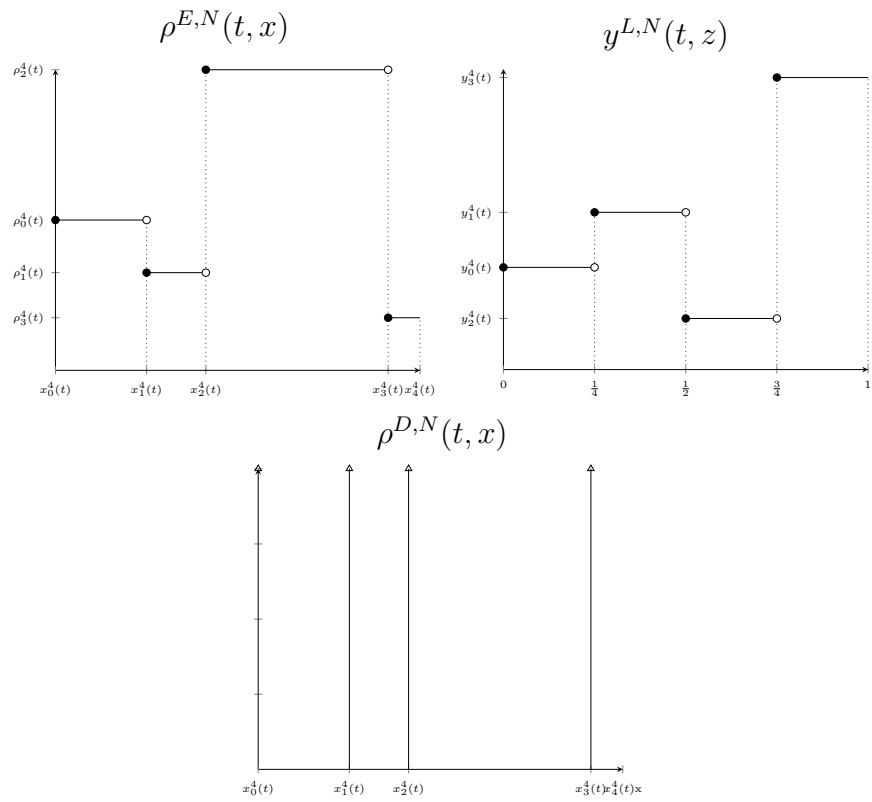

\subsection{Evolution of the supports} {In this short section, we provide a first, rough estimate about the support of all the functions defined above. The starting point is condition \eqref{cond:ubis}, that ensures the existence of a uniformly bounded initial support for the $x_i$.}

% \FR{Io userei questa proposizione un po' dappertutto in seguito. Parto da Pag. 18, ma in realtà si dovrebbe ripassare ogni volta che si usa \eqref{cond:ubis}.}

\begin{proposition} \label{p:ubisT}
    Let $\{x_j^N(t)\}_{j=0}^N$ satisfy the condition of uniformly bounded initial support \ref{cond:ubis} for some set $K\subset \R$. Let $v_{max}$ be given by \eqref{FtL-0N}. for each $T\geq 0$, define 
    \begin{eqnarray}
        \label{e-KT}
        K_T:=K+[0,T v_{max}]=\{x+z\mbox{ such that } x\in K, z\in [0,T v_{max}]\}.
    \end{eqnarray}
    It then holds
    \begin{itemize}
        \item[(i)] $x_i^N(t)\in K_T$ for all $N\in \N$, $i=0,
        \ldots, N$, $t\in[0,T]$;
        \item[(ii)] $\supp(\rho^{E,N}(t,\cdot)),\supp(y^{E,N}(t,\cdot)),\supp(\rho^{D,N}(t,\cdot))\subset  K_T$ for all $t\in[0,T]$;
        \item[(iii)] $F_{\rho^{E,N}}(t,x)=F_{\rho^{D,N}}(t,x)=0$ for all $x<\inf(K)=\inf(K_T)$ and
        
        $F_{\rho^{E,N}}(t,x)=F_{\rho^{D,N}}(t,x)=1$ for all $x>\max(K_T)=\max(K)+T v_{max}$.
    \end{itemize}
\end{proposition}
\begin{proof} Statement (i) is a direct consequence of the fact that $\dot x_i\in [0,v_\max]$ in \eqref{FtL-0}-\eqref{FtL-0N}. Statements (ii)-(iii) are then direct consequences of the definitions.
    
\end{proof}

\subsection{Convergence results for the cumulative and pseudo-inverse functions} We now recall some results about the limits of $X_{\erho}$, $X_{\drho}$, $F_{\erho}$ and $F_{\drho}$, first given in \cite{di2015rigorous}. The proofs are valid for any initial data $\{x_j^N(0)\}_{j=0}^N$ of system \eqref{ftl} that satisfies the condition of uniformly bounded initial support \eqref{cond:ubis}.

\begin{proposition}\label{prop:convofXrho}
  Let $\{x_j^N(t)\}_{j=0}^N$ be a solution of \eqref{ftl}
  %indexed by $N \in \mathbb{N}$, 
  that satisfies the condition of uniformly bounded initial support \eqref{cond:ubis}. Consider the corresponding Eulerian discrete density \linebreak $\erho \in L^{\infty}([0,+\infty) \times \mathbb{R}; [0,1])$ defined by \eqref{def:erho} and the (Dirac) empirical measure \linebreak $\drho \in L^{\infty}([0, +\infty); W_1(\mathcal{P}_c(\mathbb{R})))$ defined by \eqref{def:drho}. Let
  $F_{\erho}$,  $X_{\erho}$, $F_{\drho}$, $X_{\drho}$,
  be the corresponding 
  cumulative distributions and pseudo-inverses 
  defined by \eqref{Frho}, \eqref{pseudoinverseeul},\eqref{FDrho}, \eqref{pseudoinversedir}, respectively. Then, the following hold:
  \begin{itemize}
      \item[(i)]
      there exists a non-decreasing 
      %and $z$-right continuous 
      function 
      %$X_{\rho} \in L^{\infty}([0,T]
      $X\in L^{\infty}([0,+\infty) \times [0,1])$ such that, up to a subsequence,
   both  $\{X_{\erho}\}_{N}$ and $\{X_{\drho}\}_N$ converge to~$X$ in $L^1_{loc}([0,+\infty) \times [0,1])$;
   %for all $T>0$.
   \item[(ii)]
   define the map 
%   the map 
   $F: [0,+\infty)\times \mathbb{R} \mapsto [0,1]$ 
   as
   \begin{align}
   \label{eq:Fdef-3}
      F(t,x) \coloneqq \meas \{ z \in [0,1]: X(t,z) \leq x\},\qquad t\geq 0, \ x\in \mathbb{R},
  \end{align}
  {where $X(t,z)$ is given by statement (i).}
%  is Lipschitz continuous with respect to $x$, and,
   Then, up to a subsequence, both  $\{F_{\erho}\}_N$ and  $\{F_{\drho}\}_N$ converge to $F$ in $L^1_{\loc}([0,+\infty) \times \mathbb{R})$. 
  \end{itemize}
\end{proposition}
\begin{proof}
    See \cite[Propositions 1-2, Lemma 4]{di2015rigorous} with $L=1$, and $R=1$ (due to the maximum principle \eqref{eq:discr-maxprinc}), using their notation. 
\end{proof}
\begin{remark}
    Notice that, differently from the results in \cite{di2015rigorous}, Proposition~\ref{prop:convofXrho}
    here only states  the convergence of $\seq{X_{\erho}}$, $\seq{X_{\drho}}$ and $\seq{F_{\erho}}$, $\seq{F_{\drho}}$ up to a subsequence, which is obtained relying on Helly's compactness theorem. In \cite{di2015rigorous} the authors conclude that the whole sequences $\seq{X_{\erho}}$, $\seq{X_{\drho}}$ converge, exploiting the fact that their atomization scheme for the FtL model guarantees that $X_{\rho^{D,N+1}}(t,z) \leq X_{\rho^{D,N}}(t,z)$ for all $t\geq 0$ and $z \in [0,1]$. In turn, by the definition of the Wasserstein distance \eqref{defofwass}, 
    the convergence of the whole sequences $\seq{X_{\erho}}$ and  $\seq{X_{\drho}}$
    yields the convergence of $\seq{F_{\erho}}$ and $\seq{F_{\drho}}$.
\end{remark}

% \begin{proposition}\label{prop:convofFrho}
%   Let $\{x_j^N(t)\}_{j=0}^N$ be solutions of \eqref{ftl}, indexed by $N \in \mathbb{N}$, that satisfy the condition  of the uniformly bounded initial support \eqref{cond:ubis}. Consider the corresponding Eulerian discrete density $\erho \in L^{\infty}([0,+\infty) \times \mathbb{R})$ defined by \eqref{def:erho} and the (Dirac) empirical measure $\drho \in L^{\infty}([0,+\infty) ; W_1)$ defined by \eqref{def:drho}. Consider its corresponding cumulative distributions $F_{\erho}$ and $F_{\drho}$, defined by \eqref{Frho} and \eqref{FDrho} respectively. Then, both $\seq{F_{\erho}}$ and  $\seq{F_{\drho}}$ converge to $F_{\rho}$ in $L^1_{\loc}([0,T] \times \Omega)$ up to a subsequence, where
%   \begin{align*}
%       F_{\rho} \coloneqq \meas \{ z \in [0,1]: X_{\rho}(z) \leq x\},
%   \end{align*}
%   where $X_{\rho}(z)$ is given by Proposition \ref{prop:convofXrho}-(i).
% \end{proposition}
% \begin{proof}
%     See \cite[Proposition 2]{di2015rigorous} with $L=1$, using their notation. By the definition of the Wasserstein distance \eqref{defofwass}, the convergence of $\seq{F_{\erho}}$ and $\seq{F_{\drho}}$ depends on the convergence of $\seq{X_{\erho}}$ and  $\seq{X_{\drho}}$. By Proposition \ref{prop:convofXrho}, we get convergence up to a subsequence.
% \end{proof}

We now provide a refinement of Proposition~\ref{prop:convofXrho}-(ii). 

\begin{proposition}\label{prop:convinwassdepart} 
% Let $\{x_j^N(t)\}_{j=0}^N$ be solutions of \eqref{ftl}
% %indexed by $N \in \mathbb{N}$, 
% that satisfy the condition  of the uniformly bounded initial support~\eqref{cond:ubis}. 
Consider two sequences $\{F_{\erho}\}_N$, $\{F_{\drho}\}_N$ of cumulative distributions associated to the Eulerian discrete density $\erho$,
and to the (Dirac) empirical measure $\drho$, respectively, that converge %by Proposition~\ref{prop:convofXrho}-(ii)
to a function $F$ defined by~\eqref{eq:Fdef-3},
which is Lipschitz continuous with respect to $x$.
% of the corresponding Eulerian discrete density $\erho \in L^{\infty}([0,+\infty) \times \mathbb{R})$ defined by \eqref{def:erho} and the (Dirac) empirical measure $\drho \in L^{\infty}([0,+\infty) ;W_1)$ defined by \eqref{def:drho} such that their cumulative distributions admit a limit $F$,
% %$F_{\rho}$, 
% due to Proposition \ref{prop:convofXrho}-(ii). 
For any $t \geq 0$, let $\rho(t)$ be the distributional derivative of $x \mapsto F(t,x)$.
%F_{\rho(t)}(x)$. 
Then the following hold:
\begin{itemize}
    \item[(i)] 
    $\rho(t) \in \mathcal{P}_c(\mathbb{R})$ for all $t \geq 0$,
    \item[(ii)] $0 \leq \rho(t) \leq 1$ for almost every $t\geq 0$ and $x \in \mathbb{R}$,
    \item[(iii)] 
    $\{\erho\}_N$ and $\{\drho\}_N$ converge to $\rho$ in 
%    the topology of 
    $L_{\loc}^1\left([0,+\infty); W_{1}(\mathcal{P}_c(\mathbb{R}))\right)$.
%    , up to a subsequence.
\end{itemize}
\end{proposition}
\begin{proof}
    See \cite[Proposition 3]{di2015rigorous} with $L=1$, and $R=1$ (due to the maximum principle \eqref{eq:discr-maxprinc}), using their notation. 
    % By Proposition \ref{prop:convofXrho}-(i), we get convergence up to a subsequence.
\end{proof}
\begin{remark}
\label{rem:der-cd}
Given a map $F: [0,+\infty)\times \mathbb{R} \mapsto [0,1]$,
%defined according with Proposition~\ref{prop:convofXrho}-(ii),
%Lipschitz continuous with respect to $x$,
denote with $\rho(t)$  the distributional derivative of $x \mapsto F(t,x)$, 
and assume that $\rho(t) \in \mathcal{P}_c(\mathbb{R})$ for all $t\geq 0$.
%for all $t \geq 0$,
Then, if we consider the cumulative distribution
$F_{\rho(t)}$ as defined in~\eqref{eq:cumul-def}, one has
\begin{equation*}
    F_{\rho(t)}(x)=F(t,x)
    \quad \text{for \ a.e.} \ \ x\in \mathbb{R}.
\end{equation*}
%for all $t\geq 0$.
\end{remark}

\begin{lemma}\label{lemmaexistencerhoL}
Let $\{x_j^N(t)\}_{j=0}^N$ be a solution of \eqref{ftl},
%, indexed by $N \in \mathbb{N}$. 
and consider the Lagrangian discrete density $\lrho \in L^{\infty}([0,+\infty) \times [0,1])$ defined by \eqref{def:lrho}. Then, there exists $\rho^L \in L^{\infty}([0,T] \times [0,1])$  such that, up to a subsequence, $\seq{\rho^{L,N}}$  converges to $\rho^L$ weakly-* in $L^{\infty}([0,+\infty)\times[0,1])$.
%as $N \rightarrow \infty$.
\end{lemma}
\begin{proof}
    See \cite[Lemma 5]{di2015rigorous} with $L=1$,  and $R=1$ (due to the maximum principle \eqref{eq:discr-maxprinc}), using their notation.
\end{proof}

\section{Proof of Theorem \ref{thm:micromacrointro} 
%main result: 
}\label{sectionmicromacro}

In this section we prove the first
main result of this article, i.e. Theorem \ref{thm:micromacrointro}. 
% We deal with the following problem: on one side, let be given a solution $\rho(t)$ for the (macroscopic)  Lighthill-Whitham-Richards  model \eqref{lwr} with initial data $\bar\rho$.  On the other side, let be given a sequence of solutions $\{x^N_j(t)\}_{j=0}^N$ of the (microscopic) FtL dynamics \eqref{ftl} with initial data $\{\bar x^N_j\}_{j=0}^N$. We aim to understand which properties of the initial data and/or of the convergence of the discretized system ensure convergence of the microscopic density to the macroscopic one. The answer is given by Theorem \ref{thm:micromacrointro}.
%
With this goal, we first recall standard tools to study the Cauchy problem \eqref{lwr}: the definition of weak solution and classical results of existence and uniqueness of entropy solutions. Then, after proving a technical lemma, we present the proof of Theorem \ref{thm:micromacrointro}.
\hfill

Given the Cauchy problem \eqref{lwr}, we recall the definition of weak 
and entropy weak solution. 
\begin{definition}\label{def:weaksolution}
    A function $\rho \in L^{\infty}([0,+\infty) \times \mathbb{R})$ is a weak solution to \eqref{lwr} if it holds
    \begin{equation*}
    %\label{eq:def-ws}
\int_{\mathbb{R}} \int_{\mathbb{R}_{+}}\left[\rho(t, x) \varphi_t(t, x)+(\rho(t, x)v(\rho(t, x))) \varphi_x(t, x)\right] \mathrm{d} t \mathrm{~d} x+\int_{\mathbb{R}} \bar{\rho}(x) \varphi(0, x) \mathrm{d} x=0
\end{equation*}
for all $\varphi \in C_c^{\infty}([0,+\infty)\times \mathbb{R})$.
\end{definition}

\begin{definition}\label{def:weakentrsolution}
    A function $\rho \in L^{\infty}([0,+\infty) \times \mathbb{R})$ is a Kružkov's entropy solution to \eqref{lwr} if 
     it satisfies the entropy inequality
     \begin{equation}
         \label{eq:def-ews}
\begin{aligned}
\int_{\mathbb{R}} \int_{\mathbb{R}_{+}}[|\rho(t, x)-k| \varphi_t(t, x)+\sign(\rho(t, x)-k)&[f(\rho(t, x))-f(k)] \varphi_x(t, x)] dtdx \\
&+\int_{\mathbb{R}}|\bar{\rho}(x)-k| \varphi(0, x) dx \geq 0
\end{aligned}
     \end{equation}
for all $\varphi \in C_c^{\infty}([0,+\infty)\times \mathbb{R})$ with $\varphi$ non-negative, and for all constants $k \in \mathbb{R}$.
\end{definition}

% Note that a weak solution of \eqref{lwr} in the sense of Definition \ref{def:weaksolution}
% is also a Kružkov's entropy solution solution to \eqref{lwr} if 
%      it satisfies the entropy inequality

We now present two well-known results about the existence and uniqueness of the weak entropy solution to the Cauchy problem \eqref{lwr}.

\begin{theorem}[Uniqueness of Kružkov's solution, \cite{kruvzkov1970first}]\label{kruzkovsol}
Assume that the flux $f(\rho)$ is locally Lipschitz. For any given initial data $\bar{\rho} \in L^{\infty}$ with compact support, there exists a unique Kružkov's entropy solution $\rho \in L^{\infty}([0, +\infty) \times \mathbb{R})$ to \eqref{lwr}.

% , i.e. it satisfies the entropy inequality
% \begin{align*}
% \int_{\mathbb{R}} \int_{\mathbb{R}_{+}}[|\rho(t, x)-k| \varphi_t(t, x)+\sign(\rho(t, x)-k)&[f(\rho(t, x))-f(k)] \varphi_x(t, x)] dtdx \\
% &+\int_{\mathbb{R}}|\bar{\rho}(x)-k| \varphi(0, x) dx \geq 0
% \end{align*}
% for all $\varphi \in C_c^{\infty}([0,+\infty)\times \mathbb{R})$ with $\varphi$ non-negative and for all constants $k \in \mathbb{R}$.
\end{theorem}

\begin{theorem}[Chen and Rascle's entropy solution, \cite{chen2000initial}]\label{chenrascle}
Assume that the flux is genuinely nonlinear almost everywhere, i.e. there exists no nontrivial interval on which the flux $f(\rho)$ is affine. For a given initial data $\bar{\rho} \in L^{\infty}$ with compact support, there exists a unique weak solution $\rho \in L^{\infty}([0, + \infty) \times \mathbb{R})$ of \eqref{lwr} in the sense of Definition \ref{def:weaksolution} that satisfies the entropy inequality
\begin{align}\label{inentropychenrascle}
\int_{\mathbb{R}} \int_{\mathbb{R}_{+}}\left[|\rho(t, x)-k| \varphi_t(t, x)+\sign(\rho(t, x)-k)[f(\rho(t, x))-f(k)] \varphi_x(t, x)\right] dt dx \geq 0
\end{align}
for all $\varphi \in C_c^{\infty}((0,+\infty)\times \mathbb{R})$ with $\varphi$ non-negative and for all constants $k \in \mathbb{R}$.
Moreover, $\rho$ is the unique Kružkov's entropy solution  to \eqref{lwr}
\end{theorem}

In Theorem \ref{chenrascle} we see that, if the flux is genuinely nonlinear almost everywhere, uniqueness of entropy solution is preserved for a relaxed notion of entropy solution, which does not require the entropy inequality~\eqref{eq:def-ews}
to be satisfied at $t=0$.
This is due to the fact that 
the nonlinearity of the flux ensures 
the existence of a strong trace at $t=0$ of a
weak solution to \eqref{lwr} in the sense of Definition \ref{def:weaksolution}.
% . Notice that the required entropy inequality to be satisfied does not include the case $t=0$. Thus the definition of Chen and Rascle's entropy solution is weaker than Kružkov's.

We now present the following lemma, which is used in the proof of Theorem \ref{thm:micromacrointro}.
\begin{lemma}\label{convl1impliesconvlinf}
Consider a function $f \in L^1(\mathbb{R})\cap L^{\infty}(\mathbb{R})$ which is $1-$Lipschitz. It holds
\begin{align*}
    \norminfty{f}{\mathbb{R}} \leq \sqrt{\normone{f}{\mathbb{R}}}.
\end{align*}

\end{lemma}
\begin{proof} 

Since $|f|$ is 1-Lipschitz, for every $\bar x \in \mathbb{R}$ it holds
\begin{align*}
    |f(x)| \geq \max\{|f(\bar x)| - |x - \Bar{x}|,0\} \quad \forall \,x \in \mathbb{R}.
\end{align*}
By integrating in space, it holds
\begin{align*}
   \normone{f}{\mathbb{R}} = \int_{\mathbb{R}}|f(x)|dx \geq \int_{\mathbb{R}} \max\{|f(\bar x)| - |x - \Bar{x}|,0\} dx = |f(\bar x)|^2.
\end{align*}
Take now $\bar x_n$ such that $\lim_{n\to+\infty} |f(\bar x_n)|=\|f\|_{L^\infty(\R)}$. By passing to the limit, we have $$\normone{f}{\mathbb{R}}\geq \lim_{n\to+\infty} |f(\bar x_n)|^2=\|f\|^2_{L^\infty(\R)}.$$
\end{proof}

\medskip

We are now ready to provide the proof of the first main result of this article.

\begin{proof}[Proof of Theorem \ref{thm:micromacrointro}]
Consider the Eulerian discrete density $\erho \in L^{\infty}\big([0,+\infty) \times \mathbb{R}; [0,1]\big)$ defined by \eqref{constructionofrhoN}. To ease notation, we set
\begin{align}
\label{eq:cumdistr-short}
    {F}^N(t) \coloneqq F_{\erho(t)}, 
%    \qquad F \coloneqq F_{\rho},
\end{align}
from now on, where $F_{\erho}$ denotes the
cumulative distribution of $\erho$
given by~\eqref{Frho}. Let $F(t,x)$ be the function defined 
by~\eqref{eq:Fdef-3}, 
%$F_{\rho}$ is the 
which is equal to the 
cumulative distribution $F_{{\rho}(t)}(x)$
of its $x$-distributional derivative $\rho(t)$ (see Remark~\ref{rem:der-cd}).

\smallskip
The proof is based on two steps. 
% \textcolor{blue}{Fabio aveva scelto di non scrivere \say{Step}, potete mettervi d'accordo.}

\smallskip
% \begin{enumerate}[start=1, label= \textbf{\arabic*.}]
%     \item 
{\bf 1.}
In this step we prove that  $\seq{\erho}$, up to a subsequence, is a
Cauchy sequence in $L_{\loc}^1([              0,+\infty)\times \mathbb{R})$, under either assumption (H1) or (H2) in
%\eqref{thm:case1} and \eqref{thm:case2} of 
Theorem \ref{thm:micromacrointro}. Thus $\seq{\erho}$ converges in $L_{\loc}^1([              0,+\infty)\times \mathbb{R})$ to some limit function $\rho\in L_{\loc}^1([              0,+\infty)\times \mathbb{R})$.

    Recall by Propositions~\ref{prop:convofXrho}-\ref{prop:convinwassdepart} that, up to a subsequence, and for every $T > 0$ it holds 
    % (keeping the same notation for the subsequence)
    \begin{align}\label{convwassrho}
        \lim_{N \rightarrow +\infty} \int_0^T W_1(\erho(t),\rho(t)) dt =  \lim_{N \rightarrow +\infty} \int_0^T\normone{{F}^N(t) - F(t)}{\mathbb{R}}dt = 0.
    \end{align}
    % Then
    %  $\seq{\erho}$ 
    % is a Cauchy sequence 
    % in  $L_{\loc}^1\left([0,+\infty); W_{1}(\mathcal{P}_c(\mathbb{R}))\right)$.  
    % Consider therefore $\erho$ and $\erhoM$ for $N,M \in \mathbb{N}$ such that $M > N$.
    Since $F^N$, $F^M$ are 
    monotone non-decreasing and 
    $1-$Lipschitz in the $x$ variable, then also the function $F^N - F^M$ is $1-$Lipschitz in the $x$ variable. Therefore, by Lemma \ref{convl1impliesconvlinf} it holds
    \begin{align}\label{estimatel1linfrho}
    \quad \norminfty{F^N(t) - F^M(t)}{\mathbb{R}} \leq \sqrt{\normone{F^N(t) - F^M(t)}{\mathbb{R}}}  \quad \forall \,N,M \in \mathbb{N},\quad \forall~t>0.
    \end{align}
    It moreover holds $\supp(F^N(t) - F^M(t))\subset K_T$,
    $\supp(\erho(t) - \erhoM(t)\subset K_T$,
    for all $t\in[0,T]$, as a consequence of Proposition \ref{p:ubisT}.    
    % Omitting the time variable since no dynamics are involved, n
    % \FR{Questa ora non serve}Notice that since it holds 
    % \begin{align*}
    %     % \lim_{x \rightarrow +\infty} |F^N(t,x) - F^M(t,x)| = 
    %     \lim_{x \rightarrow \pm\infty} |F^N(t,x) - F^M(t,x)|=0, 
    %     \qquad~\forall N,M \in \mathbb{N},
    %     \quad \forall~t>0,
    % \end{align*}
    Integrating by parts, we find
    \begin{eqnarray*}
   &&\int_{\mathbb{R}}(\erho(t,x) - \erhoM(t,x) )^2dx = \int_{\mathbb{R}}\frac{d}{dx}\left(F^N(t,x) - F^M(t,x) \right)\left(\erho(t,x) - \erhoM(t,x)\right)dx\\
    &&=-\int_{\mathbb{R}}\left(F^N(t,x) - F^M(t,x)\right)\frac{d}{dx}\left(\erho(t,x) - \erhoM(t,x)\right)dx\\
    &&\leq \norminfty{F^N(t) - F^M(t)}{\mathbb{R}} \totvar \left(\erho(t) - \erhoM(t); \mathbb{R}\right)\\
    &&\leq \norminfty{F^N(t) - F^M(t)}{\mathbb{R}}\left[\totvar \left(\erho(t);\mathbb{R}\right) + \totvar \left(\erhoM(t); \mathbb{R}\right)\right].
\end{eqnarray*}
    % \FR{Io farei l'integrale su $\R$, senza cercare $\Omega$ e mettendo direttamente $K_T$.}
 By H\"older inequality
 %$L^p$ inclusion 
 and by using \eqref{estimatel1linfrho}, we thus get that 
 % for all $\Omega \subset \mathbb{R}$ bounded, 
 for all $N,M \in \mathbb{N}$, and for all $t>0$, 
 %\in [0,T]$,  
 it holds
\begin{equation}\label{inmaininomega}
    \begin{aligned}
    &\normone{\erho(t) - \erhoM(t)}{\mathbb{R}}^2 \leq \meas(K_T)\normtwo{\erho(t) - \erhoM(t)}{\mathbb{R}}^2 \\
    &\leq \meas(K_T)\norminfty{F^N(t) - F^M(t)}{\mathbb{R}}\left[\totvar \left(\erho(t); \mathbb{R}\right) + \totvar \left(\erhoM(t); \mathbb{R}\right)\right]\\
    &\leq \meas(K_T)\sqrt{\normone{F^N(t) - F^M(t)}{\mathbb{R}}}\left[\totvar \left(\erho(t); \mathbb{R}\right) + \totvar \left(\erhoM(t); \mathbb{R}\right)\right].
    \end{aligned}
\end{equation}
%     Now observe that for all $N,M \in \mathbb{N}$, and for every $T>0$,
% it holds
% \begin{align*}
%     \supp(\erho(t) - \erhoM(t)) \subset \Omega_{T,N,M}
%     \qquad \forall~t\in [0,T],
%     %(\erho, \erhoM) 
% \end{align*}
% where
% \begin{align}\label{def:omegaT1}
% %    \Omega_T(\erho, \erhoM) 
%     \Omega_{T,N,M}\coloneqq \left[\min\{x_0^N(0), x_0^M(0)\}-T, \max\{x_N^N(0), x_M^M(0)\} + T v_\max\right].
% \end{align}
% % To ease notation, in the rest of the proof we simply denote it as $\Omega_T$.
% % Thus, for all $N,M,T>0$ and for all $t \in [0,T]$ it holds
% % \begin{align}\label{randomegaT1}
% %     \normone{\erho(t) - \erhoM(t)}{\mathbb{R}} = \normone{\erho(t) - \erhoM(t)}{\Omega_T}.
% % \end{align}
% Then, applying \eqref{inmaininomega} with $\Omega = \Omega_{T,N,M}$ 
% %defined by \eqref{def:omegaT1}. Then 
% %by \eqref{randomegaT1}, 
% we deduce that for any $T>0$,
% and 
% for all $N,M,\in \mathbb{N}$, 
% %and for all $t \in [0,T]$ 
% it holds
% \begin{equation}\label{inmaininomegaT}
%     \begin{aligned}
%     &\normone{\erho(t) - \erhoM(t)}{\mathbb{R}}^2\leq  \\
%     &\leq \meas(\Omega_T)\sqrt{\normone{F^N(t) - F^M(t)}{\mathbb{R}}}\left[\totvar \left(\erho(t); \mathbb{R}\right) + \totvar \left(\erhoM(t); \mathbb{R}\right)\right]
%     \quad \forall~t\in [0,T].
%     \end{aligned}
% \end{equation}
    The further treatment of this inequality is now addressed by
%    in the following two steps, 
considering separately the two cases 
of assumption (H1) and (H2) in Theorem~\ref{thm:micromacrointro}.
    % \eqref{thm:case1} and \eqref{thm:case2} \textcolor{blue}{(Francesco dice che non si capisce, non so perché)}
    %, in the two following steps.
\smallskip

% \begin{enumerate}[start=1, label*=\textbf{\arabic*.}]
%\item \
{\bf Case (H1).} \label{proof:case11} We assume that \eqref{hyp: initial discrete TV bounded} holds. Because of the BV contractivity property enjoyed by $\erho$ and $\erhoM$ (see Proposition \ref{prop: contract Tot Var}-(i)) and relying on the hypothesis on the total variation of $\erho(0)$ and $\erhoM(0)$, it holds
    \begin{align*}
         \qquad \totvar\left(\erho(t); \mathbb{R}\right) + \totvar\left(\erhoM(t); \mathbb{R}\right) \leq  \totvar\left(\erho(0); \mathbb{R}\right) + \totvar\left(\erhoM(0); \mathbb{R}\right) \leq 2 C.
    \end{align*}
    Thus, we deduce from~\eqref{inmaininomega}
    that, for all $N,M \in \mathbb{N}$, and for all $t>0$, 
    %\in [0,T]$, 
    it holds
    \begin{align}\label{ineqforconvratelater}
    &\normone{\erho(t) - \erhoM(t)}
    {\mathbb{R}}^2
    %{\mathbb{R}}^2 
    \leq 2 C\meas(K_T)
    %(\Omega_T)
    \sqrt{\normone{{F}^N(t) - F^M(t)}{\mathbb{R}}}.
    \end{align}
     Notice that, by Hölder's inequality,
     we have
    \begin{equation}
    \label{eq:Fineq}
    \qquad\qquad \begin{aligned}
        \int_0^T \sqrt{\normone{{F}^N(t) - F^M(t)}{\mathbb{R}}}dt &\leq \normtwo{1}{[0,T]} \normtwo{\sqrt{\normone{{F}^N(t) - F^M(t)}{\mathbb{R}}}}{[0,T]}\\
    &= \sqrt{T} \sqrt{\int_0^T\normone{{F}^N(t) - F^M(t)}{\mathbb{R}}dt}.
    \end{aligned}
    \end{equation}
    Then, integrating~\eqref{ineqforconvratelater} in the time interval $[0, T]$, and using~\eqref{eq:Fineq},  we find that for all $N,M \in \mathbb{N}$ it holds
\begin{equation}
\label{eq:L1boundeulerdens-1}
\int_0^T\normone{\erho(t) - \erhoM(t)}{\mathbb{R}}^2
    %{\mathbb{R}}^2
    dt \leq 2 C\meas(K_T)
%    \meas(\Omega_T)
    \sqrt{T}\sqrt{\int_0^T\normone{{F}^N(t) - F^M(t)}{\mathbb{R}}dt}.
\end{equation}
Finally, by Hölder's inequality,
     we derive from~\eqref{eq:L1boundeulerdens-1} that
     \begin{eqnarray*}
%\label{eq:L1boundeulerdens-2}
%
    &&\int_0^T\normone{\erho(t) - \erhoM(t)}{\mathbb{R}}
    dt \leq 
    \leq \sqrt{2 C\meas(K_T)}\cdot 
    T^{\frac{3}{4}}\bigg(\int_0^T\normone{{F}^N(t) - F^M(t)}{\mathbb{R}}dt\bigg)^{\frac{1}{4}}.
\end{eqnarray*}
Therefore,  in Case (H1)
% \eqref{thm:case1} \textcolor{blue}{Francesco dice che non si capisce questo (1)} 
the convergence result \eqref{convwassrho} implies that,
for every $T>0$, 
the sequence
$\seq{\erho}$ is a Cauchy sequence in 
$L^1([0,T]\times \mathbb{R})$. 
% \FR{Per me è una sequenza in $[0,T]\times K_T$ e abbiamo convergenza in $L^1([0,T]\times K_T)$, senza necessità del loc}
% Now consider an increasing sequence $T_n>0$ and a sequence of bounded domain 
% $\Omega_n$ such that $\cup_n \big([0, T_n]\times \Omega_n\big) = [0,+\infty)\times \mathbb{R}$.
% Then, by a diagonal argument 
%  we can repeatedly extract subsequences of $\seq{\erho}$ that have the Cauchy property on each domain $[0, T_n]\times \Omega_n$, and thus we
% find a Cauchy subsequence in
% $L_{\loc}^1([0,+\infty)\times \mathbb{R})$.
\\
\smallskip

%    \item 
    {\bf Case (H2).} \label{proof:case12}  We assume that \eqref{ass:furtheronv} holds. By Proposition~\ref{prop: contract Tot Var}-(ii) and Proposition \ref{p:ubisT}, for any fixed $T, \delta > 0$, it exists a constant $C_{\delta,T}>0$ such that,
    for all $N,M \in \mathbb{N}$, it holds
    \begin{equation*}
    \sup_{t\in [\delta,T]}\left[\totvar \left(\erho(t); \mathbb{R}\right) + \totvar \left(\erhoM(t); \mathbb{R}\right)\right]\leq C_{\delta,T}\,.
    \end{equation*}
    and $\supp(\erho(t)),\supp(\erhoM(t))\subset K_T$ with $K_T$ compact, given by \eqref{e-KT}.

    With the same analysis in~\eqref{inmaininomega}, \eqref{ineqforconvratelater}, it thus follows that, for all $N,M \in \mathbb{N}$, and for all $t\in [\delta,T]$,  it holds
\begin{equation}
\label{ineqforconvratelater-2}
    \begin{aligned}
    \hspace{0.9in}
    &\normone{\erho(t) - \erhoM(t)}{\R}^2
    % {\mathbb{R}}^2 
    \\
    &\leq \meas(K_T) \sup_{t\in[\delta,T]}\left[\totvar \left(\erho(t); \mathbb{R}\right) + \totvar \left(\erhoM(t); \mathbb{R}\right)\right]\sqrt{\normone{{F}^N(t) - F^M(t)}{\mathbb{R}}},
    \\
    &\leq 2 C_{\delta,T} \meas(K_T)  \sqrt{\normone{{F}^N(t) - F^M(t)}{\mathbb{R}}}.
\end{aligned}
\end{equation}
Integrating in the time interval $[\delta, T]$ and using the Hölder's inequality as in the previous step we then find that, for all $N,M \in \mathbb{N}$, it holds
\begin{equation*}
%\label{eq:rho-deltaT-est}
\int_\delta^T\normone{\erho(t) - \erhoM(t)}{\mathbb{R}}^2  dt \leq 2C_{\delta,T}\meas(K_T)
    \sqrt{T}\sqrt{\int_\delta^T\normone{{F}^N(t) - F^M(t)}{\mathbb{R}}dt},
\end{equation*}
and
     \begin{equation}
\label{eq:L1boundeulerdens-3}
    \int_\delta^T\normone{\erho(t) - \erhoM(t)}{\mathbb{R}}
    dt  \leq \sqrt{2C_{\delta,T}\meas(K_T)}\cdot 
    \bigg(T\int_\delta^T\normone{{F}^N(t) - F^M(t)}{\mathbb{R}}dt\bigg)^{\frac{1}{4}}.
\end{equation}
% \begin{align*}
%     &\int_{\delta}^T\normone{\erho(t) - \erhoM(t)}{\Omega}^2
%     %{\mathbb{R}}^2
%     dt \\
%     &\leq \meas(\Omega_T)\sqrt{T} \sup_{t \geq \delta_1}\left[\totvar \left(\erho(t); \mathbb{R}\right) + \totvar \left(\erhoM(t); \mathbb{R}\right)\right]\sqrt{\int_{\delta_1}^T\normone{{F}^N(t) - F^M(t)}{\mathbb{R}}dt}. 
% \end{align*}
Observe now that, for any fixed $\epsilon>0$,
setting $\delta_\epsilon:= \epsilon/(2\,\meas(K_T))$,
we have
\begin{equation}
\label{eq:eq:rho-deltaT-est2}
    \int_0^{\delta_\epsilon}\normone{\erho(t) - \erhoM(t)}{\mathbb{R}}  dt \leq \frac{\epsilon}{2},\qquad\forall N,M \in \mathbb{N}\,.
\end{equation}
On the other hand,  the convergence result \eqref{convwassrho}, together with~\eqref{eq:L1boundeulerdens-3}, implies that 
there exists $N(\epsilon)>0$ such that
\begin{equation}
\label{eq:eq:rho-deltaT-est3}
    \int_{\delta_\epsilon}^T\normone{\erho(t) - \erhoM(t)}{\mathbb{R}}  dt \leq \frac{\epsilon}{2},\qquad\forall N,M\geq N(\epsilon)\,.
\end{equation}
% the sequence $\seq{\erho}$ converges in $L_{\loc}^1([\delta_1,+\infty)\times \mathbb{R})$, up to a subsequence. Denote such a subsequence corresponding to $\delta_1$ by $\{\rho^{E,N_{k}}\}_{k \in \mathbb{N}}$. It therefore holds that for all $\epsilon >0$ there exists $N(\delta_1,\epsilon)>0$ such that it holds
% \begin{align*}
%     \normone{\rho^{E,N_{n}} - \rho^{E,N_{m}}}{[\delta_1,T] \times \mathbb{R}} \leq \epsilon \qquad \forall n,m \geq N(\delta_1,\epsilon).
% \end{align*}
Therefore, combining~\eqref{eq:eq:rho-deltaT-est2}-\eqref{eq:eq:rho-deltaT-est3} we find that, also in case (H2),
%\eqref{thm:case2} 
for every $T>0$,  the sequence 
$\seq{\erho}$ is a Cauchy sequence in 
$L^1([0,T]\times \mathbb{R})$.
Then we conclude as in case~(H1).

\medskip

{\bf 2.} 
%\item
In this step we show that the function 
$\rho$ determined in the previous step
% $\seq{\erho}$ converges in $L^1_{\loc}([0,+\infty) \times \mathbb{R})$ to 
is the weak entropy solution  of the Cauchy problem \eqref{lwr}, and that 
actually the whole sequence $\seq{\erho}$ converges in $L^1_{\loc}([0,+\infty) \times \mathbb{R})$ to $\rho$.
% In the previous step, we have shown that the sequence $\seq{\erho}$ converges in $L_{\loc}^1((0,+\infty)\times \mathbb{R})$, up to a subsequence. Denote the limit function as $\rho$.

% In the previous step, by treating the cases \eqref{thm:case1} and \eqref{thm:case2}, %in \ref{proof:case11} and \ref{proof:case12} \textcolor{blue}{(Fabio aveva scelto di non usare l'environment di )}, respectively,%
% we get in both cases that the sequence $\seq{\erho}$ converges in $L_{\loc}^1([0,+\infty)\times \mathbb{R})$, up to a subsequence, to some function  $\rho\in L_{\loc}^1([0,+\infty)\times \mathbb{R})$.

Recalling that $\seq{\erho(0)}$ weakly converges to $\bar{\rho}$ by hypothesis \eqref{ass:firsthypo}, and following the same procedure as in Step 1-Case~1 of the proof of \cite[Theorem 2]{fagioli2017}, we deduce that  $\rho$ is a weak solution to \eqref{lwr} in the sense of Definition \ref{def:weaksolution}. Furthermore, it also holds that $\rho$ satisfies the entropy inequality~\eqref{inentropychenrascle}
%of Theorem \ref{chenrascle} 
by applying the exact same computations as done in the part (vi) of the proof of~\cite[Theorem 3]{di2015rigorous}. In turn, this implies that 
$\rho$ 
is a weak entropy solution  of the Cauchy problem \eqref{lwr},
thanks to Theorem~\ref{chenrascle}.
% Indeed, although in \cite{di2015rigorous} the authors use a specific discretization scheme (see \cite[(19a) and (19b)]{di2015rigorous}), the proof of \cite[Theorem 3]{di2015rigorous} does not depend on the initial discretization scheme, but rather on the dynamics of system \eqref{ftl} only. 
% \hfill
% \\\\\\
% Now we have that $\rho$ is a weak solution in the sense of Definition \ref{def:weaksolution} and it furthermore satisfies the entropy inequality \eqref{inentropychenrascle}. 
By merging Step 1 and Step 2, we conclude that,
up to a subsequence, $\seq{\erho}$ converges in $L^1_{\loc}([0,+\infty) \times \mathbb{R})$ to the unique weak entropy solution of~\eqref{lwr}.
Since, with the same arguments, we can show that any subsequence of $\seq{\erho}$ admits
a subsubsequence converging to the  unique
 weak entropy solution of~\eqref{lwr}, it follows that the whole sequence $\seq{\erho}$
 converges to $\rho$. \end{proof}

If $\bar{\rho} \in BV(\mathbb{R})$
    % \cap  L^\infty(\mathbb{R}; [0,1])$ 
    satisfies the assumptions of Theorem~\ref{thm:micromacrointro}, 
    relying on the analysis performed 
    in Step 1 of the above proof,
    one can derive the convergence rate
    for the initial Eulerian discrete density
    $\erho(0)$ associated to the atomization scheme introduced
     in~\cite[(19a)-(19b)]{di2015rigorous}. We recall it here, and prove some relevant properties.
     
\begin{proposition}\label{DiFra-Ros-scheme}
Let $\bar{\rho} \in \mathcal{P}_c(\mathbb{R})$,  with $\|\bar\rho\|_{L^\infty(\R)}\leq 1$.
    Assume that $v$ satisfies (V1).

    Define  $x^N_i(0)$ by \eqref{eq:DFR-def-infpoint}-\eqref{eq:DFR-def}. Let $\erho(0)$ be the corresponding Eulerian discrete density  at time $t=0$, defined as in~\eqref{def:erho}.
Then, the following properties hold:
\begin{itemize}
    \item[(i)]
    \begin{equation}
\label{eq:rho-l1conv}
    \rho^{E,N}(0)
    \rightarrow \bar\rho\quad \text{in}\quad L^1(\R).
\end{equation}
    \item[(ii)] If $\bar{\rho} \in BV(\mathbb{R})$, then there holds
    \begin{equation}
    \label{eq:L1-rate-conv}
        \normone{\erho(0) - \bar{\rho}}{\mathbb{R}}
        \leq \frac{C}{N^{1/4}}
    \qquad\forall~N\in\N\setminus\{0\}\,,
    \end{equation}
    for some constant $C>0$ depending on 
    $\totvar(\bar \rho;\,\mathbb{R})$ and on 
    the measure of the support of $\bar{\rho}$.
    \end{itemize}
\end{proposition}
\begin{proof}We prove (i). Observe that by definitions~\eqref{eq:DFR-def-infpoint}, \eqref{eq:DFR-def}, one has
\begin{equation}
     \label{eq:DFR-def-suppoint}
         x_N^N(0)=\bar{x}_{\max}.
     \end{equation}
%     The $L^1$ convergence~\eqref{eq:rho-l1conv} is certainly verified if $\bar\rho \in BV(\R)$ (see Remark~\ref{convergenceratescheme}). In order to show that ~\eqref{eq:rho-l1conv} holds for general 
% $\bar{\rho} \in \mathcal{P}_c(\mathbb{R})$ we  proceed as follows. Denote by $\bar{x}_{\min} < \bar{x}_{\max}$ the extremal points of the support of $\bar{\rho}$.
For any $x\in (\bar{x}_{\min},\,  \bar{x}_{\max})$, let \begin{equation*}
    I^N(x):= [x_{j_{_{N}}}^N(0), x_{j_{_{N}}+1}^N(0)),
\end{equation*}
be the interval containing $x$ for some $j_N\in\{0,\dots, N-1\}$,
and set 
\begin{equation*}
I(x):=\cap_N I^N(x).
\end{equation*}
Then we can decompose 
$(\bar{x}_{\min},\,  \bar{x}_{\max})$ as the disjoint union of the sets
\begin{equation}
\label{eq:I12decomp}
    \mathcal{I}_1:=\{x\in (\bar{x}_{\min},\,  \bar{x}_{\max})\ \big| \ I(x)=\{x\}\},\qquad\quad
     \mathcal{I}_2:=\{x\in (\bar{x}_{\min},\,  \bar{x}_{\max})\ \big| \ \{x\}\subsetneq I(x)\}\,.
\end{equation}
Notice that 
\begin{equation}
\label{eq:vanishingmeasIN}
    \lim_{N\to\infty} \meas(I^N(x))=0\qquad\forall~x\in \mathcal{I}_1.
\end{equation}
Moreover, observe that, by definitions~\eqref{eq:minimun-initial-distance}, \eqref{defofrhoj}, \eqref{def:erho}, \eqref{eq:DFR-def},
we have
\begin{equation}
\label{eq:l1-est-barrho-1}
    \rho^{E,N}(0,x)-\bar \rho(x)=\frac{1}{\meas(I^N(x))}\int_{I^N(x)}\!\!\big(\bar\rho(y)-\bar\rho(x)\big)\,dy\,,\qquad \forall~x\in (\bar x_{min},\bar x_{max}).
\end{equation}
Therefore, since $\bar\rho\in L^1(\R)$, by the Lebesgue differentiation theorem (e.g. see~\cite[\S~3.4]{folland}) we deduce from~\eqref{eq:vanishingmeasIN},
\eqref{eq:l1-est-barrho-1} that 
%for a.e. $x\in \mathcal{I}_1$ 
there holds
\begin{equation}
\label{eq:rhoEN-l1conv-1}
    \lim_{N\to\infty} \rho^{E,N}(0,x)=\bar \rho(x)
    \qquad\text{for  a.e.}\quad x\in \mathcal{I}_1\,.
\end{equation}
On the other hand, by definition~\eqref{eq:I12decomp} 
the set $\mathcal{I}_2$ is the union of intervals $J=[x_J, x'_J]$ with the property that
\begin{equation}
\label{eq:I2prop}
    [x_J, x'_J]\subset [x_{j_{_{N}}}^N(0), x_{j_{_{N}}+1}^N(0)),\qquad\forall~N,
\end{equation}
for some sequence of indices $j_N\in\{0,\dots, N-1\}$.
Hence, by definition we derive 
\begin{equation*}
    \rho^{E,N}(0,x)=\frac{1/N}{x_{j_{_{N}}+1}^N(0)-x_{j_{_{N}}}^N(0)}\leq \frac{1/N}{x'_J-x_J}\qquad\forall~x\in[x_J, x'_J],\qquad \forall~N\in\N\,,
\end{equation*}
which yields
\begin{equation}
\label{eq:rhoEN-l1conv-2}
    \lim_{N\to\infty} \rho^{E,N}(0,x)=0\qquad \forall~x\in [x_J, x'_J].
\end{equation}
Next observe that by definition~\eqref{eq:DFR-def}
and because of~\eqref{eq:I2prop}, we have
\begin{equation*}
    \int_{x_J}^{x'_J}\bar\rho(x)\,dx\leq 
    \int_{x_{j_{_{N}}}^N(0)}^{x_{j_{_{N}}+1}^N(0)}
    \bar\rho(x)\,dx=\frac{1}{N}\qquad\forall N,
\end{equation*}
which implies 
\begin{equation*}
    \int_{x_J}^{x'_J}\bar\rho(x)\,dx=0,
\end{equation*}
and thus we find
\begin{equation}
\label{eq:rhoEN-l1conv-3}
    \bar\rho(x)=0\qquad\text{for  a.e.}\quad x\in [x_J, x'_J]\,.
\end{equation}
Since $\mathcal{I}_2$ is the union of intervals of the form $[x_J, x'_J]$, we deduce from~\eqref{eq:rhoEN-l1conv-2}-\eqref{eq:rhoEN-l1conv-3} that
\begin{equation}
\label{eq:rhoEN-l1conv-4}
    \lim_{N\to\infty} \rho^{E,N}(0,x)=\bar \rho(x)
    \qquad\text{for  a.e.}\quad x\in \mathcal{I}_2\,.
\end{equation}
Then, by the dominated convergence theorem we derive
from~\eqref{eq:rhoEN-l1conv-1}, \eqref{eq:rhoEN-l1conv-4} that~\eqref{eq:rho-l1conv} is verified.\\

We now prove (ii). By the proof of~\cite[Proposition 4]{di2015rigorous} it holds
    \begin{align}
    \label{eq:w1-est-dfr}
        W_1(\erho(0), \bar{\rho}) \leq \frac{2\big(\bar{x}_{\max}-\bar{x}_{\min}\big)}{N}\qquad\forall~N\in\N\setminus\{0\}.
    \end{align}
    By using  the notation in~\eqref{eq:cumdistr-short}, this implies that $\{{F}^N(0)\}_{N\in\mathbb{N}}$ converges to $F_{\bar{\rho}}$ in $L^1(\mathbb{R})$, as $N\to\infty$.
    Moreover, by~\cite[Proposition 5]{di2015rigorous}
    we have
     $\totvar\big(\erho(0); \mathbb{R}\big)\leq \totvar\big(\bar \rho; \mathbb{R}\big)$
    for all~$N$.
    On the other hand, 
    relying on~\eqref{eq:rho-l1conv}, and
    taking the limit as $M\to\infty$
    in the inequality \eqref{inmaininomega} at $t=0$,
    with $K=[\bar{x}_{\min},\, \bar{x}_{\max}]$,
    we get  that, for all $N$, there holds
    \begin{equation}
    \label{eq:w1-est2-dfr}
        \begin{aligned}
    \normone{\erho(0) - \bar{\rho}}{\mathbb{R}} = \normone{\erho(0) - \bar{\rho}}{K}&\leq \sqrt{2 C_1 ( \bar{x}_{\max}-\bar{x}_{\min})}\normone{{F}^N(0) - F_{\bar{\rho}}}{\mathbb{R}}^{\frac{1}{4}}
    \\
    &=\sqrt{2 C_1 ( \bar{x}_{\max}-\bar{x}_{\min})}\left(W_1(\erho(0), \bar{\rho})\right)^{\frac{1}{4}},
    \end{aligned}
    \end{equation}
    where 
    $C_1=\totvar \left(\bar{\rho}; \mathbb{R}\right)$.
    Thus, combining~\eqref{eq:w1-est-dfr}-\eqref{eq:w1-est2-dfr}, we deduce~\eqref{eq:L1-rate-conv}.    
\end{proof}

\begin{remark}
Under the same assumptions of Theorem~\ref{thm:micromacrointro},
we can also
% If we assume that the initial particle positions $\{x_j^N(0)\}_{j=0}^N$ 
% belong to a bounded set $\Omega$ for all $N$ (which is a stronger assumption 
% than~\eqref{cond:ubis}), one can also
deduce that the sequence of  empirical measures $\seq{\drho}$ defined in~\eqref{def:drho}  converges in $L^1_{\loc}([0,+\infty];W_1)$ to the unique weak entropy solution $\rho$ of \eqref{lwr}. Indeed, fix $T>0$ and notice that
    \begin{align}\label{indiraceuler}
        \int_0^T W_1(\drho(t),\rho(t))dt \leq  \int_0^T W_1(\drho(t),\erho(t))dt +  \int_0^T W_1(\erho(t),\rho(t))dt.
    \end{align}
Moreover, recalling \eqref{defofwass}, \eqref{pseudoinverseeul} and \eqref{pseudoinversedir}, it holds  
\begin{eqnarray*}
            W_1(\drho(t),\erho(t))&=&\int_0^1|X^{E,N}(z)-X^{D,N}(z)|dz = \sum_{j=0}^{N-1}y_j^N\int_{jl}^{(j+1)l}[z - jl] dz 
        \\
        &&= \frac{l}{2}\sum_{j=0}^{N-1}x^N_{j+1}(t) - x^N_j(t)=\frac{l}{2}\left(x^N_N(t) - x^N_0(t)\right).
\end{eqnarray*}
   Therefore we have
   \begin{eqnarray*}
       W_1(\drho(t),\erho(t))&=&
         \int_0^T\int_0^1|X^{E,N}(t,z)-X^{D,N}(t,z)|dz \leq \frac{T(x^N_N(0) - x^N_0(0) + v_{\max}T)}{2N}.
   \end{eqnarray*}
Thanks to the condition of uniformly bounded initial support \eqref{cond:ubis}, it holds 
    \begin{align}\label{limdrhoerho}
        \sendlim{\int_0^T W_1(\drho(t),\erho(t))dt}.
    \end{align}
   % \FR{qui si dovrebbe aggiustare con Prop \ref{p:ubisT}} 
   Observe now that, 
   %
    % letting $\Omega\subset \mathbb{R}$ be a bounded set such that $\supp(\erho(t) - \rho(t)) \subset \Omega$ for all $t \in [0,T]$
    % (for instance by choosing $\Omega \supset \Omega_{T,N,M}$ where $\Omega_{T,N,M}$ is defined by~\eqref{def:omegaT1}),
%     \\
%     {\color{red} but $\Omega_{T,N,M}$ is a set that varies with $N$, its measure is uniformly bounded for all $N$ by $K_T$ in~\eqref{eq:bound-meas-omega},
%     but in general we cannot find a bounded set $\Omega\supset \Omega_{T,N,M}$ since for our scheme it may well happen that $x_0^N(0)<\overline x_{min}$
% and $x_N^N(0)>\overline x_{max}$} \\
%
    invoking Proposition \ref{p:ubisT},
    recalling \eqref{defofwass}-\eqref{Frho}, and using
    Poincaré's inequality, we have
    \begin{equation*}
        %\label{eq:w1-l1-est-1}
        W_1(\erho(t),\rho(t))=
        \normone{{F}^N(t) - F(t)}{\mathbb{R}}
        \leq C_1 \normone{\erho(t) - \rho(t)}{\mathbb{R}}
        \qquad\forall~t\in [0,T],
    \end{equation*}
    for some constant $C_1>0$. 
    Then, integrating on $[0,T]$, we derive 
    % \FR{La $T$ a destra non c'è, integri da entrambi i lati}
    % By Poincaré's inequality, for some $C>0$ independent of $N$ and for $\Omega \subset \mathbb{R}$ compact such that $\supp(\erho(t) - \rho(t)) \subset \Omega$ for all $t \in [0,T]$, for instance by choosing $\Omega = \Omega_T$ where $\Omega_T$ is defined by~\eqref{def:omegaT1}, 
    % it holds
    \begin{align}\label{inpoincare}
        \int_0^T W_1(\erho(t),\rho(t))dt \leq C_1 \normone{\erho - \rho}{[0,T] \times \mathbb{R}}.
        %\mathbb{R}}.
    \end{align}
   By merging \eqref{indiraceuler},
   \eqref{limdrhoerho}, \eqref{inpoincare} and Theorem \ref{thm:micromacrointro}, it holds
    \begin{align*}
        \sendlim{   \int_0^T W_1(\drho(t),\rho(t))dt}
        \qquad\forall~T>0.
    \end{align*}
\end{remark}

The next Proposition shows how to 
define an atomization scheme
$\{\tilde{x}_j^N(0)\}_{j=0}^N$
different from the one in~\cite{di2015rigorous}, 
whose corresponding initial Eulerian discrete density
 satisfies the assumption~\eqref{ass:firsthypo} of~Theorem~\ref{thm:micromacrointro}.
Thus, if the velocity function satisfies (V2),
according with Theorem~\ref{thm:micromacrointro}, 
the scheme $\{\tilde x_j^N(0)\}_{j=0}^N$ leads to an Eulerian discrete density $\erhotilde(t,x)$ which
 still converges in $L^1_{loc}$ 
to the weak entropy solution of the Cauchy problem \eqref{lwr}. 
On the other hand, we also show that 
the initial Eulerian discrete density associated to  $\{\tilde x_j^N(0)\}_{j=0}^N$ does not satisfy assumption (H1) of 
Theorem~\ref{thm:micromacrointro}.
Hence, in this case one would not expect that
$\erhotilde(t,x)$  converges 
to the weak entropy solution of the Cauchy problem \eqref{lwr}. However, we will discuss
in Remark~\ref{rem:FtLnot convergingtoLWR}
some numerical simulations that seem to suggest 
that such convergence holds for $\erhotilde(t,x)$ too.
% as we will discuss in Remark~\ref{rem:FtLnot convergingtoLWR}.

% Recalling the notation in Definition~\eqref{defofFandX} we have the following result concerning the discretized densities associated to $\{\widetilde x_j^N(0)\}_{j=0}^N$
% at time $t=0$.

\begin{proposition}\label{p:differentscheme-f}
    Let $\bar{\rho} \in \mathcal{P}_c(\mathbb{R})$, with $\|\bar\rho\|_{L^\infty(\R)}\leq 1/2$.
    %\mathcal{P}_c(\mathbb{R}; [0,1])$, be with compact support.
    Assume that $v$ satisfies (V1).
    Define  $\{x_j^N(0)\}_{j=0}^N$ as in \eqref{eq:DFR-def-infpoint}-\eqref{eq:DFR-def}.
    % and let $\rho^{E,N}(0)
    % $, $\rho^{L,N}(0)$
    % be the corresponding Eulerian and Lagrangian discrete densities at time $t=0$, defined as in~\eqref{def:erho}, \eqref{def:lrho}, respectively.
    % Consider 
   Also define $\{\tilde x_j^N(0)\}_{j=0}^N$
   as follows
\begin{equation}
\label{defof: tilde x_j}
    \tilde x_j^N(0)=\begin{cases}
    x_j^N(0)&\mbox{~~if~} \ j\mbox{~is~even or }j=N,\\
    \noalign{\medskip}
    \dfrac{x_{j-1}^N(0)+x_j^N(0)}2&\mbox{~~if~} \ j<N\mbox{~is~odd}.
\end{cases}
\end{equation}
Let $\tilde\rho^{E,N}(0)$ be the corresponding Eulerian discrete density at time $t=0$, defined as in~\eqref{def:erho}.
Then, the following properties hold:
% Define $$\rho^{E,N}=\sum_{j=0}^{N-1} \frac{1/N}{x^N_{j+1} - x^N_j} \chi_{[x^N_j, x^N_{j+1})}(x)\qquad\mbox{~~and~~}\qquad
% \nu^{E,N}=\sum_{j=0}^{N-1} \frac{1/N}{\tilde x^N_{j+1} - \tilde x^N_j} \chi_{[\tilde x^N_j, \tilde x^N_{j+1})}(x).$$
\begin{itemize}
    \item[(i)] $\tilde \rho^{E,N}(0)  \rightharpoonup \bar{\rho}
 \ \ \text{weak\,}^*\ \ \text{in}\ \ L^\infty(\mathbb{R})$;
 %    \begin{align}
 % &\tilde \rho^{E,N}(0)  \rightharpoonup \bar{\rho}
 % \qquad\text{weak\,}^*\ \ \text{in}\ \ L^\infty(\mathbb{R}),
 % \label{tilde discret converges weakly}
 % \\
 % \noalign{\medskip}
 %   &\tilde\rho^{L,N}(0)-\rho^{L,N}(0) \nrightarrow
 %   0\quad \text{in}\quad L^1([0,1]), 
 %   \label{no-lagr-conv}
 %   \\
 %    \noalign{\medskip}
 %   &\erhotilde(0) \nrightarrow
 %   \bar{\rho}\qquad \text{in}\quad L^1(\mathbb{R}).
 %   \label{no eul conv}
 %    \end{align}
    \item[(ii)] 
%    $\erhotilde(0) 
$\tilde\rho^{E,N_k}(0)\nrightarrow
  \bar{\rho}\ \  \text{in}\quad L^1(\mathbb{R})$,
  for every subsequence $\{\tilde\rho^{E,N_k}(0)\}_k$\,;
    \item[(iii)] The sequence $TV\big(\erhotilde(0)\big)$ is unbounded.
\end{itemize}
% It then holds:
% \begin{enumerate}
%     \item $\rho^{E,N},\nu^{E,N}\rightharpoonup \bar \rho$ weak$^*$ in $L^\infty(\mathbb{R})$;
%     \item $\rho^{E,N}-\nu^{E,N}\not\to 0$ in $L^1(\mathbb{R})$;
%     \item if $\bar\rho\in BV(\R)$, the following results hold:
%     \begin{enumerate}
%         \item \label{st-nuEN-f} $\nu^{E,N}\not\to \bar \rho$ in $L^1(\mathbb{R})$;
%         \item \label{st-nuLN-f} $\nu^{L,N}-\rho^{L,N} \not\to 0$ in $L^1(\mathbb{R})$.   
%         \item The sequence $TV(\nu^{E,N})$ is unbounded
%     \end{enumerate}
\end{proposition}
\begin{proof} Denote with $\lfloor a \rfloor$ the integer part of $a$ and define $N':=\lfloor\frac{N}{2}\rfloor-1$.

We first prove (i).   By~\eqref{eq:DFR-def}, \eqref{defof: tilde x_j}, 
and because of definitions~\eqref{eq:minimun-initial-distance}, \eqref{defofrhoj}, \eqref{def:erho}, 
we have 
\begin{equation}
\label{eq:int between-x2j-1}
    \int_{x_{2j}^N(0)}^{x_{2j+2}^N(0)}  \tilde \rho^{E,N}(0,x)\,dx=\int_{x_{2j}^N(0)}^{x_{2j+2}^N(0)} \bar \rho(x)\,dx =\frac{2}{N} \qquad\  \forall~j=0,\ldots,N'.
\end{equation}
% \begin{align}\label{int between x2j and x2j+2 even is 2/N}
% \int_{\tilde x_{2j}^N(0)}^{\tilde x_{2j+2}^N(0)} \bar \rho(x)dx = \int_{x_{2j}^N(0)}^{x_{2j+2}^N(0)} \bar \rho(x)dx = \frac{2}{N} \qquad\  \forall~j=0,\ldots,\left\lfloor\frac{N}{2}\right\rfloor-1,
% \end{align}
% while definitions~\eqref{defofrhoj}, \eqref{def:erho}, yield
% \begin{align}
% \label{eq:int between x2-2}
%     \int_{\tilde x_{2j}^N(0)}^{\tilde x_{2j+2}^N(0)}  \tilde \rho^{E,N}(0,x)dx=\int_{x_{2j}^N(0)}^{ x_{2j+2}^N(0)}  \rho^{E,N}(0,x)dx = \frac{2}{N}\qquad \ \forall~j=0,\ldots,\left\lfloor\frac{N}{2}\right\rfloor-1.
% \end{align}
In the same way, if $N$ is odd,
and thus $2\lfloor \frac{N}{2}\rfloor = N-1,$
we find
\begin{equation*}
%\label{int between xN-1}
\int_{x_{N-1}^N(0)}^{x_{N}^N(0)} \tilde\rho^{E,N}(0,x)\,dx =\int_{x_{N-1}^N(0)}^{x_{N}^N(0)} \bar \rho(x)\,dx
= \frac{1}{N}.
\end{equation*}
% \begin{equation}
% \label{int between xN-1 and xN odd is 1/N}
% \begin{aligned}
%     \int_{\tilde x_{N-1}^N(0)}^{\tilde x_{N}^N(0)} \bar \rho(x)dx &= \int_{x_{N-1}^N(0)}^{x_{N}^N(0)} \bar \rho(x)dx = \frac{1}{N},
%     \\
%     \noalign{\smallskip}
%     \int_{\tilde x_{N-1}^N(0)}^{\tilde x_{N}^N(0)} \tilde \rho^{E,N}(0,x)dx &= \int_{x_{N-1}^N(0)}^{x_{N}^N(0)} \rho^{E,N}(0,x)dx = \frac{1}{N}.
% \end{aligned}
% \end{equation}
% $$\int_{x_{2j}^N}^{x_{2j+2}^N}\bar\rho\,dx=\int_{x_{2j}^N}^{x_{2j+2}^N}\rho^{E,N}\,dx=\int_{\tilde x_{2j}^N}^{\tilde x_{2j+2}^N}\nu^{E,N}\,dx=\frac{2}N.$$
% Moreover, if $N$ is odd, it also holds 
% $$\int_{x_{N-1}^N}^{x_{N}^N}\bar\rho\,dx=\int_{x_{N-1}^N}^{x_{N}^N}\rho^{E,N}\,dx=\int_{\tilde x_{N-1}^N}^{\tilde x_{N}^N}\nu^{E,N}\,dx=\frac{1}N.$$
Consider a test function $\varphi\in C^\infty_c(\R)$,
and set
% . It is $L$-Lipschitz, for some $L>0$. We denote with 
$$a_j^N:=\min_{x\in[x_{2j}^N,x_{2j+2}^N]}\varphi(x),\qquad
b_j^N:=\max_{x\in[x_{2j}^N,x_{2j+2}^N]}\varphi(x)
\qquad \quad \forall~j=0,\ldots,N'.$$
%for $j\leq \frac{N-1}2$. 
By monotonicity of the integral (i.e. $\varphi\leq \psi\Rightarrow \int \varphi\rho\,dx\leq \int \psi\rho \,dx$ for $\rho\geq 0$), it follows from~\eqref{eq:int between-x2j-1}
%, \eqref{int between xN-1} 
that
\begin{equation}
\label{eq:intjxrhoel}
\begin{aligned}
    \int_{x_{2j}^N(0)}^{x_{2j+2}^N(0)} \varphi(x)\bar \rho(x)\,dx 
    % = \int_{x_{2j}^N(0)}^{x_{2j+2}^N(0)} \varphi(x)\bar \rho(x)\,dx 
    \in\frac{2}N[a_j^N,b_j^N],\qquad\quad 
    % \int_{x_{2j}^N(0)}^{x_{2j+2}^N(0)} \varphi(x)\rho^{E,N}(0,x)\,dx\in\frac{2}N[a_j^N,b_j^N],\qquad
\int_{x_{2j}^N(0)}^{x_{2j+2}^N(0)}\varphi(x)\tilde \rho^{E,N}(0,x)\,dx\in\frac{2}N[a_j^N,b_j^N].
\end{aligned}
\end{equation}
%for all $j=0,\dots, \lfloor\frac{N}{2}\rfloor-1$.
Moreover, letting $L$ be the Lipschitz costant of $\varphi$, it holds $b_j^N-a_j^N\leq L (x_{2j+2}^N-x_{2j}^N)$. As a consequence, we derive from~\eqref{eq:intjxrhoel},  that
\begin{equation*}
        \left|\int_{x_{2j}^N(0)}^{x_{2j+2}^N(0)} \varphi(x)\big(\bar \rho(x)-\tilde\rho^{E,N}(0,x)\big)\,dx \right|\leq \frac{2\,L (x_{2j+2}^N-x_{2j}^N)}N  \qquad\ 
         \forall~j=0,\ldots,N'.
    \end{equation*}
% \begin{equation}
%     \int_{x_{2j}^N}^{x_{2j+2}^N}|f(\bar\rho-\rho^{E,N})|\,dx\leq \frac{2}N L (x_{2j+2}^N-x_{2j}^N),\qquad \int_{\tilde x_{2j}^N}^{\tilde x_{2j+2}^N}|f(\bar\rho-\nu^{E,N})|\,dx\leq \frac{2}N L (\tilde x_{2j+2}^N-\tilde x_{2j}^N).
% \end{equation}
For $N$ even, this  implies
\begin{equation} 
 \left|\int_{\R} \varphi(x)\big(\bar \rho(x)-\tilde\rho^{E,N}(0,x)\big)\,dx \right|
\leq \frac{2\,L}N  \sum_{j=0}^{\frac{N}{2}-1}(x_{2j+2}^N-x_{2j}^N)=\frac{2\, L\, (x_N^N-x_0^N)}N .
\label{e-sommatoria-weak-f}
\end{equation}
% and similarly $\int_\R|f(\bar\rho-\nu^{E,N})|\,dx\leq \frac{2}N L (\tilde x_N^N-\tilde x_0^N)
% $. 
For $N$ odd, one needs to consider also the additional estimate
\begin{equation*}
        \left|\int_{x_{N-1}^N(0)}^{x_{N}^N(0)} \varphi(x)\big(\bar \rho(x)-\tilde\rho^{E,N}(0,x)\big)\,dx \right|\leq \frac{L (x_{N}^N-x_{N-1}^N)}N.
\end{equation*}
%
% $$\int_{x_{N-1}^N}^{x_{N}^N}|f(\bar\rho-\rho^{E,N})|\,dx\leq \frac{1}N L (x_{N}^N-x_{N-1}^N),\qquad \int_{\tilde x_{N-1}^N}^{\tilde x_{N}^N}|f(\bar\rho-\nu^{E,N})|\,dx\leq \frac{1}N L (x_{N}^N-x_{N-1}^N),$$
that anyway leads to \eqref{e-sommatoria-weak-f}. In both cases, recalling~\eqref{eq:DFR-def-infpoint},
\eqref{eq:DFR-def-suppoint},
we deduce from~\eqref{e-sommatoria-weak-f} 
that
\begin{equation*} 
%\label{e-sommatoria-weak-f-2}
 \left|\int_{\R} \varphi(x)\big(\bar \rho(x)-\tilde\rho^{E,N}(0,x)\big)\,dx \right|
\leq \frac{2\, L\, (\bar{x}_{\max}-\bar{x}_{\min})}N,
\end{equation*}
% this proves $\rho^{E,N},\nu^{E,N}\rightharpoonup \bar\rho$. 
which proves (i).
\smallskip

We now prove (ii). 
% \FR{Sono riuscito a togliere il caso dispari}
%{\bf Item (2).} 
% By considering a subsequence of $\{\tilde\rho^{E,N}(0)\}_{N\in\mathbb{N}}$
% we may assume that $N$ is even.
%Assume that $N$ is even.
In view of Proposition~\ref{DiFra-Ros-scheme}-(i), 
in order to establish (ii) it will be sufficient to show that
\begin{equation}
\label{eq:rhotrho-nol1conv}
    \tilde\rho^{E,N_k}(0)-\rho^{E,N_k}(0)
    \nrightarrow 0\quad \text{in}\quad L^1(\R),
\end{equation}
 for every subsequence $\{\tilde\rho^{E,N_k}(0)\}_k$.
To this end, denote with $\rho^{E,N}(0)
    $
 the  Eulerian  discrete densities at time $t=0$, defined as in~\eqref{def:erho} in connection with the scheme $\{x_j^N(0)\}_{j=0}^N$  in \eqref{eq:DFR-def-infpoint}-\eqref{eq:DFR-def}. By \eqref{eq:minimun-initial-distance}, \eqref{defofrhoj}, and \eqref{defof: tilde x_j},
    we have
    \begin{equation*}
       \tilde x_{2j+2}^N(0)-\tilde x_{2j+1}^N(0)=
       \frac{x_{2j+2}^N(0)-x_{2j+1}^N(0)}{2},
    \end{equation*}
    and
\begin{equation*}
    \rho^{E,N}(0,x)=\frac{1/N}{x^N_{2j+2}(0) - x^N_{2j+1}(0)},\qquad\quad
     \tilde\rho^{E,N}(0,x)=\frac{2/N}{x^N_{2j+2}(0) - x^N_{2j+1}(0)},
\end{equation*}
for all $x\in [\tilde x_{2j+1}^N(0),\,\tilde x_{2j+2}^N(0)],$ and for all $j=0,\dots,  N'$.
As a consequence we find
\begin{equation*}
    \int_{\tilde x_{2j+1}^N(0)}^{\tilde x_{2j+2}^N(0)}\Big|\tilde\rho^{E,N}(0,x)-\rho^{E,N}(0,x)\Big|\,dx=
    %\frac12(x^N_{2j+2} - x^N_{2j+1})\frac{2/N-1/N}{x^N_{2j+2} - x^N_{2j+1}}=
    \frac1{2N},
    \qquad \quad \forall~j=0,\ldots,N',
\end{equation*}
which yields
\begin{eqnarray*}
    \big\|\tilde\rho^{E,N}(0)-\rho^{E,N}(0)\big\|_{L^1(\R)}
    &\geq& \sum_{j=0}^{N'}\int_{\tilde x_{2j+1}^N(0)}^{x_{2j+2}^N(0)}\Big|\tilde\rho^{E,N}(0,x)-\rho^{E,N}(0,x)\Big|\,dx
    \\
    &=&\frac1{2N}(N'+1)\geq \frac{N-1}{4N}>\frac18\qquad\quad \forall~N>2\,.
\end{eqnarray*}
%
% We consider the subsequence with $N$ even. Take an interval $[x_{2j}^N,x_{2j+2}^N]$. By construction, it holds $x_{2j+1}^N<\tilde x_{2j+1}^N<x_{2j+2}^N$. As a consequence, for $x\in [\tilde x_{2j+1}^N,x_{2j+2}^N]$ it holds $\rho^{E,N}(x)=\frac{1/N}{x^N_{2j+2} - x^N_{2j+1}}$ and $\nu^{E,N}(x)=\frac{1/N}{\tilde x^N_{2j+2} - \tilde x^N_{2j+1}}=\frac{2/N}{x^N_{2j+2} - x^N_{2j+1}}$. It then holds 
% $$\int_{\tilde x_{2j+1}^N}^{\tilde x_{2j+2}^N}(\nu^{E,N}(x)-\rho^{E,N}(x))\,dx=\frac12(x^N_{2j+2} - x^N_{2j+1})\frac{2/N-1/N}{x^N_{2j+2} - x^N_{2j+1}}=\frac1{2N}.$$
% It thus holds $$\|\nu^{E,N}-\rho^{E,N}\|_{L^1}\geq \sum_{j=0}^{N/2-1}\int_{\tilde x_{2j+1}^N}^{x_{2j+2}^N}(\nu^{E,N}(x)-\rho^{E,N}(x))\,dx=\frac1{2N}\frac{N}{2}=\frac14>0.$$
%
This implies \eqref{eq:rhotrho-nol1conv},
% In view of~\eqref{eq:rhotrho-nol1conv}, in order to prove (ii) it will be sufficient to show~\eqref{eq:rho-l1conv} 
 thus completing the proof of (ii). 
\smallskip

We finally prove (iii). By contradiction, assume that there exists a subsequence (that we do not relabel) such that  $TV\big(\erhotilde(0)\big)$ is uniformly bounded. Since $\tilde\rho^{E,N}(0)$
is uniformly bounded in $L^\infty(\mathbb{R})$, by Helly's compactness theorem there exists a further subsequence (that we do not relabel) which converges in $L^1(\mathbb{R})$ to some function
$\tilde\rho$. Statement  (i) ensures that $\bar \rho= \tilde \rho$, that is a contradiction with (ii). This proves (iii). \end{proof}

\bigskip

\begin{remark}
    \label{rem:FtLnot convergingtoLWR}
    Given   $\bar{\rho} \in \mathcal{P}_c(\mathbb{R})$,
    consider  the atomization scheme
$\{\tilde x_j^N(0)\}_{j=0}^N$
defined in~\eqref{defof: tilde x_j}, and let $\{\tilde x^N_j(t)\}_{j=0}^{N}$ be the solution of \eqref{ftl}, with traffic velocity given by
\begin{equation}
\label{eq:v-bm}
    v(\rho)=\begin{cases}
        \exp{\left(\dfrac{\rho}{\rho-1}\right)}\  &\text{if} \ \ \rho<1,
        \\
        \noalign{\smallskip}
        \ 0 \   &\text{if} \ \ \rho = 1,
    \end{cases} 
\end{equation}
according with the Bonzani and Mussone's model~\cite{bonz-muss}. It is clear that the associated flux is not strictly concave, see Figure \ref{fig:flux exp}.
\begin{figure}[!htb]
    \centering
    \includegraphics[width=0.6\textwidth]{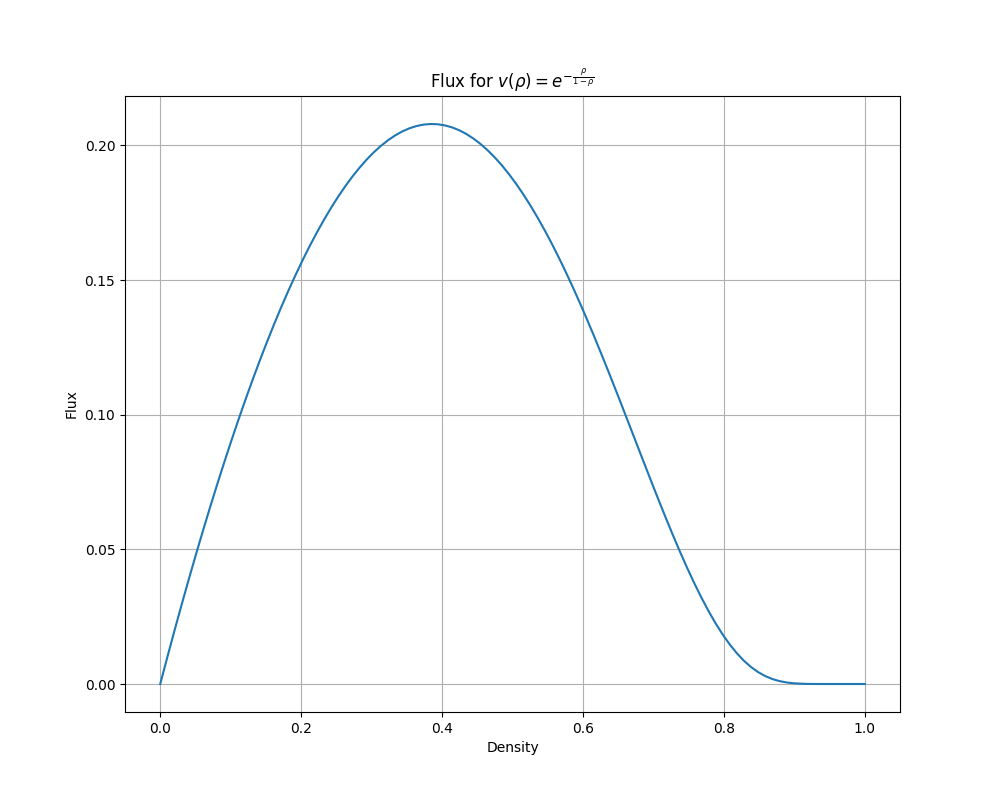}
    \caption{Flux associated with \eqref{eq:v-bm}.}
    \label{fig:flux exp}
\end{figure}

More precisely, it holds $\rho\, v'(\rho)=-\dfrac{\rho\,v(\rho)}{(1-\rho)^2}$, which is a decreasing function in the interval $[0,\,\frac{\sqrt{5}-1}{2}]$, and an increasing function in the interval
$[\frac{\sqrt{5}-1}{2},\, 1]$. Therefore, the velocity $v(\rho)$ in~\eqref{eq:v-bm} does not satisfy the assumption~\eqref{ass:furtheronv}.

On the other hand, we have shown in Proposition~\ref{p:differentscheme-f}-(iii) that, letting 
$\tilde\rho^{E,N}(0)$
    be the  Eulerian  discrete densities at time $t=0$
    corresponding to $\{\tilde x_j^N(0)\}_{j=0}^N$, one has that $TV\big(\erhotilde(0)\big)$ is unbounded. Hence, neither condition (H1) nor (H2)
    of Theorem~\ref{thm:micromacrointro} are satisfied.
    Thus, although $\tilde\rho^{E,N}(0)$ satisfies the
    assumption~\eqref{ass:firsthypo} of Theorem~\ref{thm:micromacrointro} (as shown in Proposition~\ref{p:differentscheme-f}-(i)), 
    one cannot expect that the sequence 
  $\seq{\erhotilde}$ converges 
to the weak entropy solution $\rho$ of the Cauchy problem \eqref{lwr}, in general.

To investigate this behaviour, we discuss here the numerical simulations corresponding to the discretization scheme $\{\tilde x^N_j(t)\}_{j=0}^{N}$,
with velocity $v(\rho)$ in~\eqref{eq:v-bm},
and initial datum\begin{equation*}
    \bar\rho(x) = \frac{1}{2}\chi_{[\frac{1}{2},\,\frac{5}{2}]}(x)\,.
\end{equation*}

The parameters for the numerical simulation are chosen as follows. We set the space step-size $\Delta x = 0.01$, time step-size $\Delta t = 0.001$ and the time period $[0,3]$. Since the system becomes stiffer as $N$ increases, we choose an implicit method to numerically solve the system. In particular, we use an implicit method based on backward-differentiation formulas (BDF) of automatically-varying order (from 1 to 5), already implemented in the Python library \say{scipy} as ‘BDF'. The general framework of such algorithm is described in \cite{byrne1975polyalgorithm} and the Python implementation follows a quasi-constant step size as explained in \cite{shampine1997matlab}.

In Figure \ref{fig:combined}, we show snapshots of the evolution of the two different profiles $\rho^{E,5}(t)$ and $\tilde \rho^{E,5}(t)$. It is remarkable that for some initial short period of time, the fluxes experimented by both profiles are visibly distinct, see Figure \ref{fig:snapshots dynamics in [0,0.23]}. Yet, after some further time, both profiles start exhibiting an extremely similar behavior, (e.g., starting at $t=1$ in Figure \ref{fig:snapshots dynamics in [0,2]}). Indeed, an approximation of a rarefaction wave can be seen on the right and a shock on the left. They then interact, thus causing the decrease of the total variation at the interaction point. This corresponds to the behavior at the macroscopic level. Afterwards, both profiles remain then in the same configuration with decreasing density with respect to $x$, i.e. $\erho(x_1) \leq \erho(x_2)$ with $x_2 \leq x_1$. This property is indeed preserved forward in time by the FtL. 

In Figure \ref{fig:difference evol for different N}, we show the evolution in time of the $L^1$ distance (in the space variable) between $\erho(t)$ and $\tilde\erho(t)$, for different values of $N$. The striking phenomenon is that the $L^1$ distance is a decreasing function of time, with a very strong decay as $N$ increases. This suggests that the following result might hold for this initial datum: for every $t>0$ and for arbitrary $\epsilon>0$, there exists $N > 0$ such that
\begin{align*}
    \int_0^t \normone{\erho(t) - \erhotilde(t)}{\mathbb{R}} \leq \epsilon.
\end{align*}
In particular, it seems that the discrepancy between the approximating solutions is arbitrarily small for arbitrarily small time. In other terms, Theorem \ref{thm:micromacrointro} might be non-sharp and a more general result of convergence might be available.

% \begin{figure}[!htb]
%     \centering
%     \includegraphics[width=0.6\textwidth]{evolution_plot_bothdensities separated_N_5 v=exponential T=3.0.png}
%     \caption{Caption}
%     \label{fig:enter-label}
% \end{figure}

% \begin{figure}[!htb]
%     \centering
%     \includegraphics[width=0.6\textwidth]{evolution_plot_bothdensities together_N_5 v=exponential T=3.0.png}
%     \caption{Caption}
%     \label{fig:enter-label}
% \end{figure}

\begin{figure}[htbp]
    \centering
    \begin{subfigure}[b]{0.49\textwidth}
        \centering
        \includegraphics[width=\textwidth]{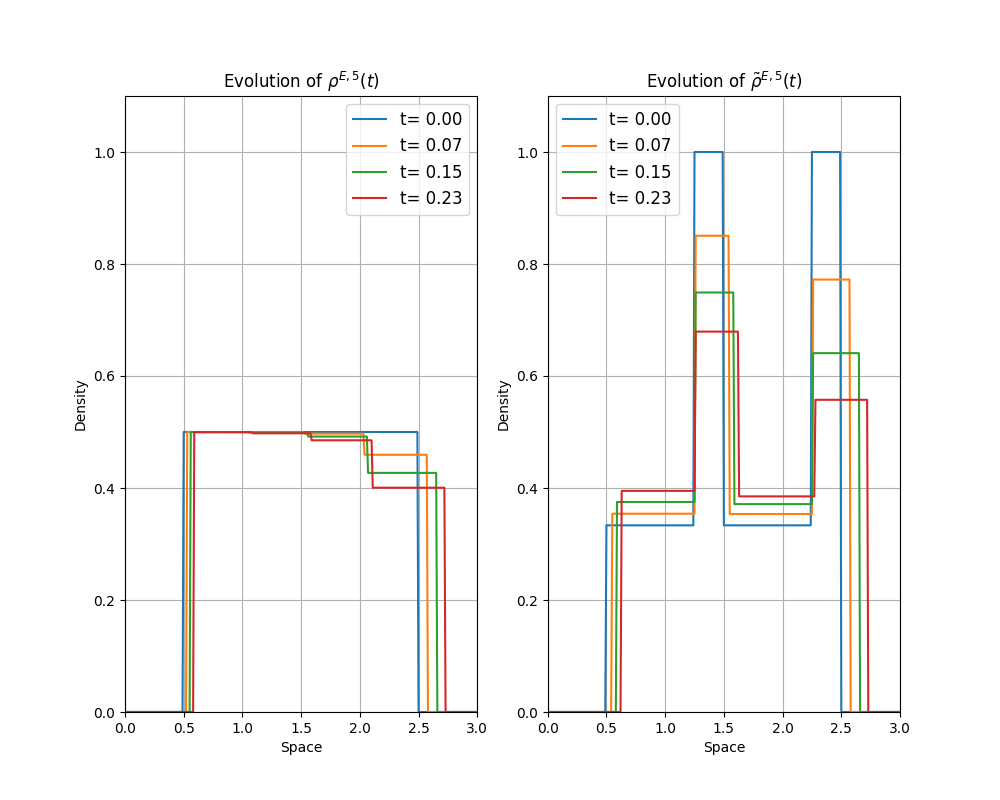}
        \caption{Evolution in the time period $[0,0.23]$}
        \label{fig:snapshots dynamics in [0,0.23]}
    \end{subfigure}
    \begin{subfigure}[b]{0.49\textwidth}
        \centering
        \includegraphics[width=\textwidth]{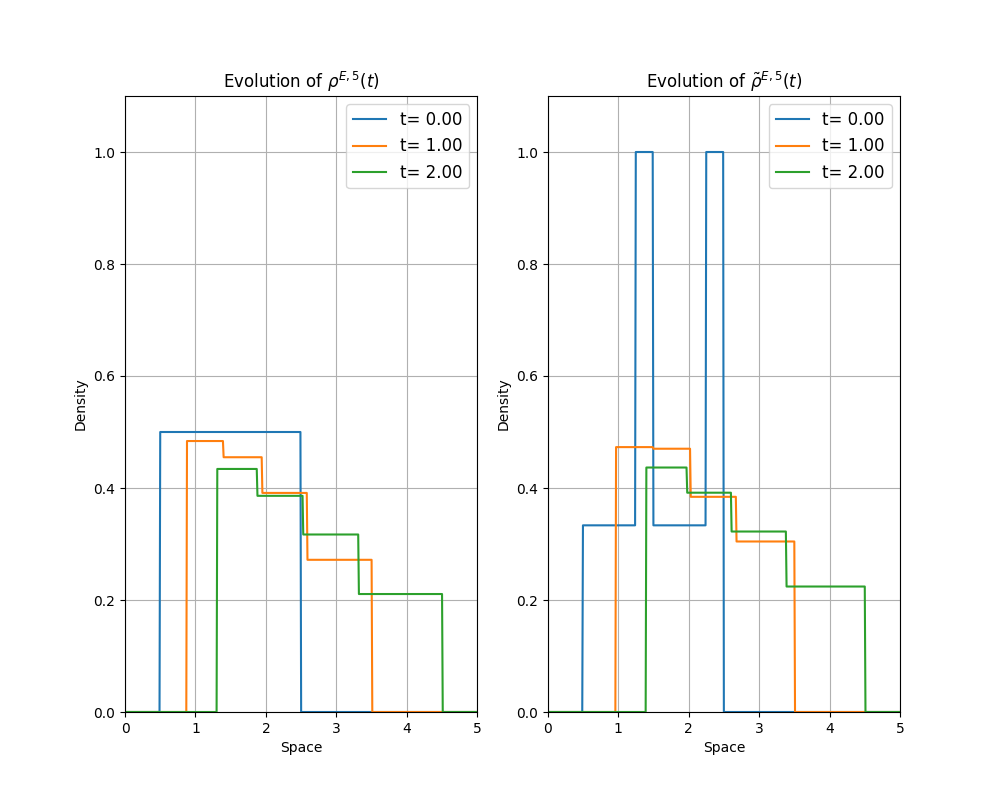}
        \caption{Evolution in the time period $[0,2]$}
        \label{fig:snapshots dynamics in [0,2]}
    \end{subfigure}
    \caption{Snapshots of the dynamics of $\rho^{E,5}(t)$ and $\tilde \rho^{E,5}(t)$.} 
    \label{fig:combined}
\end{figure}

\begin{figure}[!htb]
    \centering
    \includegraphics[width=0.9\textwidth]{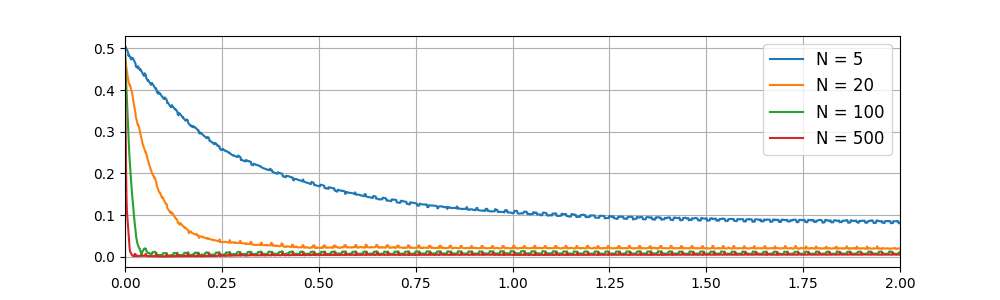}
    \caption{Evolution of  $\|\rho^{E,N}(t)-\tilde \rho^{E,N}(t)\|_{L^1(\R)}$ for $N=5,20,100,500$.}
    \label{fig:difference evol for different N}
\end{figure}

\end{remark}

\section{Proof of Theorem \ref{thm:stabilitytheoremintro}}
%Stability in Eulerian coordinates}
\label{sectionmainstability}

In this section, we prove Theorem \ref{thm:stabilitytheoremintro}, which provides a stability result for two different Eulerian discrete densities, in both the Wasserstein  and the $L^1$ norm. 
%In the rest of the paper, 
Here, we compare two solutions $\{x^N_j(t)\}_{j=0}^{N}$, $\{\tilde{x}^N_j(t)\}_{j=0}^{N}$ of the FtL model \eqref{ftl} and the corresponding Eulerian discrete densities $\erho$, $\erhotilde$ defined by \eqref{def:erho}.

We first state three propositions, which lead to the proof of the main theorem.
\begin{proposition}\label{lagrangianimplieswasserstein}
      Given two sequences $\{x_j^N\}_{j=0}^N$,$\{\tilde{x}_j^N\}_{j=0}^N$, 
      %indexed by $N \in \mathbb{N}$,
      satisfying conditions~\eqref{eq:initial-vehicles}-\eqref{eq:minimun-initial-distance},
      assume that $x_N^N = \tilde{x}_N^N$. Consider the corresponding Eulerian densities $\erho$, $\erhotilde\in L^{\infty}(\mathbb{R})$ defined by \eqref{def:erho}. Then, it holds 
   \begin{align}
   \label{eq:w1euldensest1}
       W_1(\rho^{E,N}, \tilde{\rho}^{E,N}) \leq 2\sum_{j=0}^{N-1}|x_{j+1} - x_j - (\tilde{x}_{j+1} - \tilde{x}_j) |
       %\qquad\forall~N. 
   \end{align}
    
\end{proposition}
\begin{proof}
For any fixed  $j =0,..., N-1$, and 
for every $z \in [jl, (j+1)l)$, recalling the definition of $y_j^N$ in \eqref{def:defofyj}, we have
\begin{align*}
     \left|x^N_j - \tilde{x}^N_j + (z-jl)\left(y^N_j  - \tilde{y}^N_j\right)\right| &= \left|x^N_j - \tilde{x}^N_j + \frac{z - jl}{l}\left(x^N_{j+1} - \tilde{x}^N_{j+1}  - (x^N_j - \tilde{x}^N_j)\right)\right| \\
    &\leq \left|x^N_j - \tilde{x}^N_j\right| + \left|x^N_{j+1} - \tilde{x}^N_{j+1}  - (x^N_j - \tilde{x}^N_j)\right|\\
    &\leq 2\left|x^N_j - \tilde{x}^N_j - (x^N_{j+1} - \tilde{x}^N_{j+1} )\right|+ \left|x^N_{j+1} - \tilde{x}^N_{j+1}\right| \\
    &\leq 2\left(\sum_{k=j}^{N-1}\left|x^N_k - \tilde{x}^N_k - (x^N_{k+1} - \tilde{x}^N_{k+1} )\right|\right)+ \left|x^N_{N} - \tilde{x}^N_{N}\right|,
\end{align*}
where in the last inequality we repeatedly make use of the triangular inequality 
\begin{align*}
    \left|x^N_k - \tilde{x}^N_k\right| \leq \left|x^N_k - \tilde{x}^N_k - (x^N_{k+1} - \tilde{x}^N_{k+1} )\right| +\left|x^N_{k+1} - \tilde{x}^N_{k+1}\right| \qquad k=j+1,\ldots,N-1.
\end{align*}
Therefore, by summing in $j$, 
and since $x_N^N = \tilde{x}_N^N$, it holds
\begin{equation}
    \label{eq:sumj}
    \begin{aligned}
     \sum_{j=0}^{N-1}\left|x^N_j - \tilde{x}^N_j + (z-jl)\left(y^N_j  - \tilde{y}^N_j\right)\right| &\leq 2\left(\sum_{j=0}^{N-1}\sum_{k=j}^{N-1}\left|x^N_k - \tilde{x}^N_k - (x^N_{k+1} - \tilde{x}^N_{k+1} )\right|\right)
     + N\left|x^N_{N} - \tilde{x}^N_{N}\right|
     \\
    &\leq 2N\left(\sum_{j=0}^{N-1}\left|x^N_j - \tilde{x}^N_j - (x^N_{j+1} - \tilde{x}^N_{j+1} )\right|\right)+0\\
    % + N\left|x^N_{N} - \tilde{x}^N_{N}\right|\\
    &=2\sum_{j=0}^{N-1}\left|y^N_j -\tilde{y}^N_j\right|.
    %+ N\left|x^N_{N} - \tilde{x}^N_{N}
    %\right|.
\end{aligned}
\end{equation}
Then, relying on~\eqref{eq:sumj}, and recalling~\eqref{pseudoinverseeul},
we find
\begin{align*}
    \int_0^1|X_{\rho^{E,N}}(z)-X_{\tilde{\rho}^{E,N}}(z)|dz &= \int_{0}^{1}\sum_{j=0}^{N-1}\left|x^N_j - \tilde{x}^N_j + (z-jl)\left(y^N_j  - \tilde{y}^N_j\right)\right|\chi_{[jl, (j+1)l)}(z)dz \\
    &\leq \int_{0}^{1}
    % \left[2\sum_{j=0}^{N-1}\left|y^N_j -\tilde{y}^N_j\right| + N\left|x^N_{N} - \tilde{x}^N_{N}\right|\right]
    2\sum_{j=0}^{N-1}\left|y^N_j -\tilde{y}^N_j\right| 
    \chi_{[jl, (j+1)l)}(z)dz = 2\normone{y^{L,N} - \tilde{y}^{L,N} }{[0,1]}\,.
    %&=\frac{2}{N}\sum_{j=0}^{N-1}\left|y^N_j -\tilde{y}^N_j\right| %+ \left|x^N_{N} - \tilde{x}^N_{N}\right|
\end{align*}
By \eqref{eq:Ldensl1norm}, \eqref{defofwass},  this proves~\eqref{eq:w1euldensest1}.
\end{proof}

\begin{proposition}\label{wassconv}
Assume that the velocity map $v$ satifies \eqref{e-V1}.
  Let $\{x_j^N(t)\}_{j=0}^N$, $\{\tilde{x}_j^N(t)\}_{j=0}^N$ be solutions of \eqref{ftl}, 
  %indexed by $N \in \mathbb{N}$ 
  that satisfy the  condition of uniformly bounded initial support  \eqref{cond:ubis}. Consider the corresponding Eulerian discrete densities $\rho^{E,N}$, $\tilde{\rho}^{E,N} \in L^{\infty}([0,+\infty) \times \mathbb{R})$ defined by \eqref{def:erho}. Then, for any fixed $T > 0$ and for all $N$, it holds
  \begin{equation}
      \label{eq:w1-est-23}
      \begin{aligned}
    \sup_{t \in [0,T]}W_1\big(\rho^{E,N}(t), \tilde{\rho}^{E,N}(t)\big) &\leq W_1\big(\rho^{E,N}(0), \tilde{\rho}^{E,N}(0)\big) + 
    \\
    &+2LT\sum_{j=0}^{N-1}|x_{j+1}(0) - x_j(0) - (\tilde{x}_{j+1}(0) - \tilde{x}_j(0)) |,
\end{aligned}
  \end{equation}
where $L$ is the Lipschitz constant of~$v$.
\end{proposition}

\begin{proof}[Proof.  ]
% \begin{enumerate}[start=1, label={\bfseries{Step} \arabic*.}]

% \item 
{\bf 1.}
In this step we show that, for $z \in [1-l,1)$, it holds
\begin{equation}
\label{eq:particle-est-1}
\begin{aligned}
&\left|x^N_{N-1}(t) - \tilde{x}^N_{N-1}(t) + (z-1+l)\left(\frac{1}{\rho^N_{N-1}(t)}  - \frac{1}{\tilde{\rho}^N_{N-1}(t)}\right)\right|\leq  \\
        &\qquad\quad 
        \leq\left|x^N_{N-1}(0) - \tilde{x}^N_{N-1}(0) + (z-1+l)\left(\frac{1}{\rho^N_{N-1}(0)}  - \frac{1}{\tilde{\rho}^N_{N-1}(0)}\right)\right| +
        \\
    &\qquad\qquad
    + L\int_0^t \left|\rho^N_{N-1}(s) - \tilde{\rho}^N_{N-1}(s)\right| ds,
\end{aligned}
\end{equation}
and, for all $j = 0,\ldots,N-2$, and  $z \in [jl,(j+1)l)$, it holds
\begin{equation}
\label{eq:particle-est-2}
    \begin{aligned}
        &\left|x^N_j(t) - \tilde{x}^N_j(t) + (z-jl)\left(\frac{1}{\rho^N_j(t)}  - \frac{1}{\tilde{\rho}^N_j(t)}\right)\right|\leq 
        \\
        &\qquad\quad \leq \left|x^N_j(0) - \tilde{x}^N_j(0) + (z - jl)\left(\frac{1}{\rho^N_j(0)}  - \frac{1}{\tilde\rho^N_j(0)}\right)\right|+ \\
    &\qquad\qquad+
    L \int_0^t\Big(\left|\rho^N_j(s) - \tilde{\rho}^N_j(s)\right| + \left|\rho^N_{j+1}(s) - \tilde{\rho}^N_{j+1}(s)\right|\Big)ds.
    \end{aligned}
\end{equation}
To this end,  first notice that \eqref{ftl} 
ensures
\begin{align}
\label{eq:particle-eq-1}
    x^N_N(t)- \tilde{x}_N^N(t) = x^N_N(0)- \tilde{x}_N^N(0),
\end{align}
%while for $j=0,\ldots,N-1$ it holds
\begin{align}
\label{eq:particle-eq-2}
    x^N_j(t)- \tilde{x}_j^N(t) = x^N_j(0)- \tilde{x}_j^N(0) + \int_0^t v(\rho_j^N(t)) - v(\tilde{\rho}_j^N(t))dt,
    \qquad j=0,\ldots, N-1.
\end{align}
Moreover, observe that
%the fact that
\begin{equation*}
%\label{eq:z-est-1}
    1 - \frac{z-1+l}{l} \leq 1 \qquad \forall \,z \in [1-l,1],
\end{equation*}
and that, recalling~\eqref{defofrhoj}, we have the identity
\begin{equation}
\label{eq:particle-est-3}
\begin{aligned}
     &\big(x^N_{N-1}(t) - \tilde{x}^N_{N-1}(t)\big)\left(1 - \frac{z-1+l}{l}\right) + \frac{z-1+l}{l} \big(x^N_N(t) - \tilde{x}^N_N(t)\big)= \\
    &\hspace{2in} 
    = x^N_{N-1}(t) - \tilde{x}^N_{N-1}(t) + (z-1+l)\left(\frac{1}{\rho^N_{N-1}(t)}  - \frac{1}{\tilde{\rho}^N_{N-1}(t)}\right).
\end{aligned}
\end{equation}
Then, relying on~\eqref{eq:particle-eq-1}-\eqref{eq:particle-est-3}, 
and using  the Lipschitz continuity of the velocity $v$, we derive that,
for 
%$j=N-1$ and 
$z \in [1-l,1)$, it holds
\begin{align*}
    &\left|x^N_{N-1}(t) - \tilde{x}^N_{N-1}(t) + (z-1+l)\left(\frac{1}{\rho^N_{N-1}(t)}  - \frac{1}{\tilde{\rho}^N_{N-1}(t)}\right)\right| \\
    &= \left|(x^N_{N-1}(t) - \tilde{x}^N_{N-1}(t))\left(1 - \frac{z-1+l}{l}\right) + \frac{z-1+l}{l} (x^N_N(t) - \tilde{x}^N_N(t))\right|\\
    &= \left|\left(x^N_{N-1}(0) - \tilde{x}^N_{N-1}(0) + \int_0^t (v(\rho^N_{N-1}(s)) - v(\tilde{\rho}^N_{N-1}(s)) )ds\right)\left(1 - \frac{z-1+l}{l}\right) \right. \\
    &\quad+ \left.\frac{z-1+l}{l}\left(x^N_N(0) - \tilde{x}^N_N(0) \right)\right|\\
     &\leq \left|(x^N_{N-1}(0) - \tilde{x}^N_{N-1}(0))\left(1 - \frac{z-1+l}{l}\right) + \frac{z-1+l}{l} (x^N_N(0) - \tilde{x}^N_N(0))\right| \\
     &\quad + \left(1 - \frac{z-1+l}{l}\right)\int_0^t |v(\rho^N_{N-1}(s)) - v(\tilde{\rho}^N_{N-1}(s)) |ds\\
     &\leq\left|x^N_{N-1}(0) - \tilde{x}^N_{N-1}(0) + (z-1+l)\left(\frac{1}{\rho^N_{N-1}(0)}  - \frac{1}{\tilde{\rho}^N_{N-1}(0)}\right)\right|  + L\int_0^t \left|\rho^N_{N-1}(s) - \tilde{\rho}^N_{N-1}(s)\right| ds,
\end{align*}
% where in the last inequality we have used the Lipschitz continuity of the velocity,
which proves~\eqref{eq:particle-est-1}.

Next, observe that, for $j=0,...,N-2$, one has
\begin{align}
\label{eq:z-est-2}
    1 - \frac{z - jl}{l} \leq 1 \qquad \text{and} \qquad \frac{z - jl}{l} \leq 1 \qquad \forall \,z \in [jl, (j+1)l],
\end{align}
and  the identity
\begin{equation}
\label{eq:particle-est-4}
\begin{aligned}
    &(x^N_j(t) - \tilde{x}^N_j)(t)\left(1 - \frac{z - jl}{l}\right) + \frac{z - jl}{l} \big(x^N_{j+1}(t) - \tilde{x}^N_{j+1}(t)\big) \\
    &\hspace{2in} 
    = x^N_j(t) - \tilde{x}^N_j(t) + (z - jl)\left(\frac{1}{\rho^N_j(t)}  - \frac{1}{\tilde{\rho}^N_j(t)}\right).
\end{aligned}
\end{equation}
Then, relying on~\eqref{eq:particle-eq-2}, \eqref{eq:z-est-2}, \eqref{eq:particle-est-4}, with similar computations as above we find that, for $j=0,...,N-2$, and for $z \in [jl, (j+1)l)$, it holds
\begin{eqnarray*}
    &&\left|x^N_j(t) - \tilde{x}^N_j(t) + (z-jl)\left(\frac{1}{\rho^N_j(t)}  - \frac{1}{\tilde{\rho}^N_j(t)}\right)\right| \\
     && =\left|(x^N_j(0) - \tilde{x}^N_j(0))\left(1 - \frac{z - jl}{l}\right) + \frac{z - jl}{l} \big(x^N_{j+1}(0) - \tilde{x}^N_{j+1}(0)\big) +\right.\\
    &&\qquad\ + \!\!\left. \left( 1 - \frac{z - jl}{l}\right)\int_0^t \Big(v(\rho^N_{j}(s)) - v(\tilde{\rho}^N_{j}(s))\Big)ds\! + \!\frac{z - jl}{l}\int_0^t \Big(v(\rho^N_{j+1}(s)) - v(\tilde{\rho}^N_{j+1}(s)) \Big)ds\right|\\
    &&\leq \left|(x^N_j(0) - \tilde{x}^N_j(0))\left(1 - \frac{z - jl}{l}\right) + \frac{z - jl}{l} (x^N_{j+1}(0) - \tilde{x}^N_{j+1}(0)) \right|+\\
    && \qquad\ +  L \int_0^t\Big(\left|\rho^N_j(s) - \tilde{\rho}^N_j(s)\right| + \left|\rho^N_{j+1}(s) - \tilde{\rho}^N_{j+1}(s)\right|\Big)ds\\
    &&\leq \left|x^N_j(0) - \tilde{x}^N_j(0) + (z - jl)\left(\frac{1}{\rho^N_j(0)}  - \frac{1}{\tilde\rho^N_j(0)}\right)\right| + \\
    &&\qquad\ +
    L \int_0^t\Big(\left|\rho^N_j(s) - \tilde{\rho}^N_j(s)\right| + \left|\rho^N_{j+1}(s) - \tilde{\rho}^N_{j+1}(s)\right|\Big)ds,
\end{eqnarray*}
which proves~\eqref{eq:particle-est-2}.
% where in the last inequality we have used the Lipschitz continuity of the velocity, the fact that
% Therefore, it holds
% \begin{align*}
%     &\sum_{j=0}^{N-2}\left|x^N_j(t) - \tilde{x}^N_j(t) + (z-jl)\left(\frac{1}{\rho^N_j(t)}  - \frac{1}{\tilde{\rho}^N_j(t)}\right)\right| \\
%     &\qquad + \left|x^N_{N-1}(t) - \tilde{x}^N_{N-1}(t) + (z-1+l)\left(\frac{1}{\rho^N_{N-1}(t)}  - \frac{1}{\tilde{\rho}^N_{N-1}(t)}\right)\right| \\
%     &\leq \sum_{j=0}^{N-2}\left|x^N_j(0) - \tilde{x}^N_j(0) + (z-jl)\left(\frac{1}{\rho^N_j(0)}  - \frac{1}{\tilde{\rho}^N_j(0)}\right)\right| \\
%     &\qquad + \left|x^N_{N-1}(0) - \tilde{x}^N_{N-1}(0) + (z-1+l)\left(\frac{1}{\rho^N_{N-1}(0)}  - \frac{1}{\tilde{\rho}^N_{N-1}(0)}\right)\right| + 2L\int_0^t\sum_{j=0}^{N-1}\left|\rho^N_j(s) - \tilde{\rho}^N_j(s)\right|\\
%     &\leq \sum_{j=0}^{N-2}\left|x^N_j(0) - \tilde{x}^N_j(0) + (z-jl)\left(\frac{1}{\rho^N_j(0)}  - \frac{1}{\tilde{\rho}^N_j(0)}\right)\right|ds \\
%     &\qquad + \left|x^N_{N-1}(0) - \tilde{x}^N_{N-1}(0) + (z-1+l)\left(\frac{1}{\rho^N_{N-1}(0)}  - \frac{1}{\tilde{\rho}^N_{N-1}(0)}\right)\right| + 2Lt\sum_{j=0}^{N-1}\left|y^N_j(0) - \tilde{y}^N_j(0)\right|
% \end{align*}
% where in the last inequality we have used Proposition \ref{contractinv}.
\medskip

%    \item 
{\bf 2.}
By definition of the pseudo-inverse given in \eqref{pseudoinverseeul}, using the bounds~\eqref{eq:particle-est-1}, \eqref{eq:particle-est-2} and recalling~\eqref{eq:minimun-initial-distance},
%Proposition \ref{contractinvrho},
we get
\begin{eqnarray*}
    &&\int_0^1|X_{\rho^{E,N}(t)}(z)-X_{\tilde{\rho}^{E,N}(t)}(z)|dz \\
    &&= \int_{0}^{1}\sum_{j=0}^{N-1}\left|x^N_j(t) - \tilde{x}^N_j(t) + (z-jl)\left(\frac{1}{\rho^N_j(t)}  - \frac{1}{\tilde{\rho}^N_j(t)}\right)\right|\chi_{[jl, (j+1)l)}(z)dz \\
    &&\leq \int_0^1\sum_{j=0}^{N-1}\left|x^N_j(0) - \tilde{x}^N_j(0) + (z-jl)\left(\frac{1}{\rho^N_j(0)}  - \frac{1}{\tilde{\rho}^N_j(0)}\right)\right|\chi_{[jl, (j+1)l)}(z)dz\\
    &&\qquad +   2L\int_0^1\int_0^t\sum_{j=0}^{N-1}\left|\rho^N_j(s) - \tilde{\rho}^N_j(s)\right|\chi_{[jl, (j+1)l)}(z)dsdz,
    \\
    &&= \int_0^1\sum_{j=0}^{N-1}\left|x^N_j(0) - \tilde{x}^N_j(0) + (z-jl)\left(\frac{1}{\rho^N_j(0)}  - \frac{1}{\tilde{\rho}^N_j(0)}\right)\right|\chi_{[jl, (j+1)l)}(z)dz\\
    &&\qquad +   \frac{2L}{N}\int_0^t\sum_{j=0}^{N-1}\left|\rho^N_j(s) - \tilde{\rho}^N_j(s)\right|ds.
    %\chi_{[jl, (j+1)l)}(z)dsdz
    \end{eqnarray*}
Then, applying Proposition \ref{contractinvrho} 
we deduce
 \begin{align}
 \label{eq:pi-est-2}
    \int_0^1|X_{\rho^{E,N}(t)}(z)-X_{\tilde{\rho}^{E,N}(t)}(z)|dz \leq \normone{X_{\rho^{E,N}}(0)-X_{\tilde{\rho}^{E,N}}(0)}{[0,1]} +   \frac{2Lt}{N}\sum_{j=0}^{N-1}\left|y^N_j(0) - \tilde{y}^N_j(0)\right|.
    \end{align}   
 Notice that, by~\eqref{eq:minimun-initial-distance}, \eqref{def:ly}, we have
\begin{align}
\label{eq:le-eq}
    \normone{y^{L,N}(0) - \tilde{y}^{L,N}(0)}{[0,1]} =
    \sum_{j=0}^{N-1} |y_j^{N}(0) - \tilde{y}_j^{N}(0)| \int_0^1 \chi_{[jl,(j+1)l)}(z) dz = \frac{1}{N}\sum_{j=0}^{N-1} |y_j^{N}(0) - \tilde{y}_j^{N}(0)|.
\end{align}
Therefore, from~\eqref{eq:pi-est-2}-\eqref{eq:le-eq}, we obtain
\begin{equation}
\label{eq:pi-ineq23}
\begin{aligned}
    \normone{X_{\rho^{E,N}}(t)-X_{\tilde{\rho}^{E,N}}(t)}{[0,1]} &\leq \normone{X_{\rho^{E,N}}(0)-X_{\tilde{\rho}^{E,N}}(0)}{[0,1]} + \\
    &\quad +2Lt\normone{y^{L,N}(0)- \tilde{y}^{L,N}(0)}{[0,1]},
    \qquad \forall~t>0.
\end{aligned}
\end{equation}
By Definition~\ref{def:wassersteindistance}, we can restate~\eqref{eq:pi-ineq23} in terms of of the Wasserstein distance as 
\begin{align}
\label{eq:w1-ineq23}
    W_1(\rho^{E,N}(t), \tilde{\rho}^{E,N}(t)) \leq W_1(\rho^{E,N}(0), \tilde{\rho}^{E,N}(0)) + 2Lt\normone{y^{L,N}(0)- \tilde{y}^{L,N}(0)}{[0,1]},
    \qquad \forall~t>0.
\end{align}
Taking the supremum in~\eqref{eq:w1-ineq23} over  the time interval $[0,T]$, and recalling~\eqref{eq:Ldensl1norm},  we recover the inequality~\eqref{eq:w1-est-23}.
%\end{enumerate}
\end{proof}

\begin{proposition}\label{convinfimpliesconvl1}
Under the same assumptions of Proposition~\ref{wassconv}, 
for any $T>0$ the following hold:
\begin{itemize}
  \item[(i)] if there exists $C > 0$ %independent of $N$ 
    such that
    $\totvar \left(\erho(0); \mathbb{R}\right),\totvar \left(\erhotilde(0); \mathbb{R}\right) < C$ for all $N$, then %for all $T > 0$ 
    there exist $C_T>0$ such that
    for all $N$
    there holds
\begin{equation*}
     \quad \normone{\erho(t) - \erhotilde(t)}
    {\mathbb{R}}^2 
    \leq 2\, C_T
    \sqrt{\sup_{t\in[0,T]}\normone{F_{\erho}(t) - F_{\erhotilde}(t)}{\mathbb{R}}},
    \qquad \forall~t\in [0,T];
%    ,\quad \forall~N;
\end{equation*}
      \item[(ii)] if the velocity $v$ satisfies \eqref{ass:furtheronv}, then %for all $T > 0$ 
      for any $\delta> 0$,
      there exist $C_\delta, C_{\delta,T}>0$
      such that for all $N$ there hold
      \begin{equation*}
          \totvar \left(\erho(t); \mathbb{R}\right)< C_\delta,
          \qquad 
          \totvar \left(\erhotilde(t); \mathbb{R}\right)< C_\delta\qquad \forall~t\geq\delta,          %\quad \forall~N,
      \end{equation*}
      and
\begin{equation*}
      \quad \
      \normone{\erho(t) - \erhotilde(t)}
    {\mathbb{R}}^2 
    \leq 2\,  C_{\delta,T}
    \sqrt{\sup_{\tau\in[0,T]}\normone{F_{\erho}(\tau) - F_{\erhotilde}(\tau)}{\mathbb{R}}},
    \qquad \forall~t\in [\delta,T].
%    ,\quad \forall~N.
\end{equation*}
\end{itemize}
\end{proposition}
\begin{proof} 
The proofs 
%of (1) and (2) 
can be obtained with precisely the same arguments of 
Step~1 of the proof of Theorem \ref{thm:micromacrointro}, replacing $\erhoM, F^M$, with
$\erhotilde, \tilde F^N$, 
respectively: see~\eqref{ineqforconvratelater}, \eqref{ineqforconvratelater-2}.
\end{proof}
\medskip

We are now ready to establish the proof of the second main result of this paper.
\begin{proof}[Proof of Theorem \ref{thm:stabilitytheoremintro}]
The proof is based on the concatenation of the above propositions.
Namely, notice that the estimate~\eqref{eq:wasse-stab}
is an immediate consequence of 
Proposition \ref{wassconv}.
Also, observe that, by  Definition~\ref{def:wassersteindistance}, we have
\begin{equation*}
    \normone{F_{\erho}(t) - F_{\erhotilde}(t)}{\mathbb{R}}=W_1\big(\rho^{E,N}(t), \tilde{\rho}^{E,N}(t)\big).
\end{equation*}
Applying Proposition \ref{lagrangianimplieswasserstein}
together with Proposition~\ref{convinfimpliesconvl1}, and relying on~\eqref{stabilityhypo}, we derive
the uniform limits~\eqref{unifconv-bvbdd}
and
\begin{equation}
\label{unifconv-bvinf-delta}
      \sendlim{\sup_{\, t \in [\delta,\, T]}\normone{\rho^{E,N}(t) - \tilde{\rho}^{E,N}(t)}{\mathbb{R}}},
      \qquad\forall~\delta>0.
\end{equation}
Then, we recover the uniform limit \eqref{unifconv-bvinf} from the limit in~\eqref{unifconv-bvinf-delta}
with $\delta=1/k$, for some subsequences $\{{\rho}^{E,N_k}\}_k$\,, $\{\tilde{\rho}^{E,N_k}\}_k$ 
constructed
by a diagonal procedure. 
This completes the prof of the theorem.
\end{proof}

% \begin{remark}
% \label{rem:L1conv-euler-discrete}
% \end{remark}

% \FR{Nella biblio, metti solo le iniziali dei nomi propri. Inoltre, togli le informazioni di troppo, come ho fatto io nel Folland.}

\bibliographystyle{plain}
\bibliography{references}

% \textcolor{blue}{Nelle referenze, non capisco cosa vuole Francesco nella citazione 10 di Elio-Emanuela-Federico. Immagino che chiederà la stessa cosa nella nuova referenza 9 di Holden-Risebro.}
\end{document}